\documentclass[10pt]{amsart}

\usepackage[dvipsnames, rgb]{xcolor}
\usepackage{amssymb, amsmath, amsthm, amsfonts, mathrsfs}
\usepackage{graphicx}
\usepackage{tikz}
\usepackage{tikz-cd}
\usetikzlibrary{shapes,arrows}
\usepackage{float}
\usepackage{hvfloat}
\usepackage[width=\textwidth]{caption}
\usepackage{rotating}
\usepackage{adjustbox}

\usepackage{url}
\usepackage[all,arc,2cell]{xy}
\UseAllTwocells
\usepackage{enumerate}

\usepackage{multicol}

\usepackage{hyperref}
  \definecolor{dark-red}{rgb}{0.6,0.15,0.15}
   \definecolor{dark-blue}{rgb}{0.15,0.15,0.6}
   \definecolor{medium-blue}{rgb}{0,0,0.5}
   
   \definecolor{cbf-bluedark}{RGB}{17, 112, 170}
           \definecolor{cbf-blue}{RGB}{95, 162, 206}
               \definecolor{cbf-bluepale}{RGB}{163, 204, 233}
               
                        \definecolor{cbf-bluepale2}{RGB}{162, 200, 236}
        
      \definecolor{cbf-orangedark}{RGB}{200, 82, 0}
          \definecolor{cbf-orange}{RGB}{252, 125, 11}
      \definecolor{cbf-orangepale}{RGB}{255, 188, 121}

\hypersetup{
    colorlinks, 
    linkcolor=dark-red,
    citecolor=dark-blue, urlcolor=medium-blue
}

\usepackage[nameinlink,capitalise,noabbrev]{cleveref}

\numberwithin{equation}{section}

\newtheorem{theorem}{Theorem}[section]

\newtheorem{cor}{Corollary}[section]
\newtheorem{corollary}{Corollary}[section]

\newtheorem{prop}{Proposition}[section]

\newtheorem{lem}{Lemma}[section]
\newtheorem{lemma}{Lemma}[section]

\theoremstyle{definition}
\newtheorem{defn}{Definition}[section]
\newtheorem{definition}{Definition}[section]

\newtheorem{rem}{Remark}[section]
\newtheorem{warn}{Warning}[section]
\newtheorem*{warn*}{Warning}

\makeatletter
\let\c@equation=\c@thm
\let\c@lem=\c@thm
\let\c@theorem=\c@thm
\let\c@lemma=\c@thm
\let\c@Theorem=\c@thm
\let\c@Lemma=\c@thm
\let\c@cor=\c@thm
\let\c@corollary=\c@thm
\let\c@Corollary=\c@thm
\let\c@conj=\c@thm
\let\c@conjecture=\c@thm
\let\c@prop=\c@thm
\let\c@proposition=\c@thm
\let\c@Proposition=\c@thm
\let\c@defn=\c@thm
\let\c@definition=\c@thm
\let\c@Definition=\c@thm
\let\c@notation=\c@thm
\let\c@note=\c@thm
\let\c@exmp=\c@thm
\let\c@ex=\c@thm
\let\c@exmps=\c@thm
\let\c@rem=\c@thm
\let\c@warn=\c@thm
\let\c@claim=\c@thm
\let\c@convention=\c@thm
\let\c@conventions=\c@thm
\let\c@quest=\c@thm
\let\c@facts=\c@thm
\let\c@slogan=\c@thm
\let\c@ass=\c@thm
\makeatother

\newcommand{\F}{\mathbb{F}}

\newcommand{\Z}{\mathbb{Z}}
\newcommand{\W}{\mathbb{W}}
\newcommand{\Q}{\mathbb{Q}}
\newcommand{\G}{\mathbb{G}}

\newcommand{\FF}{\mathbb{F}}

\newcommand{\GG}{\mathbb{G}}

\def\SS{\mathbb{S}}

\newcommand{\xra}{\xrightarrow}

\def\makeop#1{\expandafter\def\csname #1\endcsname{\mathop{\mathrm{#1}}\nolimits}}

\makeop{id}
\makeop{Mod}
\makeop{Hom}
\makeop{Tot}
\makeop{gr}
\makeop{Out}
\makeop{Hom}
\makeop{Ext}
\makeop{End}
\makeop{Aut}
\makeop{Tor}
\makeop{ev}
\makeop{hocolim}
\makeop{holim}
\makeop{map}
\makeop{id}
\makeop{colim}
\makeop{im}
\makeop{pic}
\newcommand{\WW}{\mathbb{W}}

\usepackage{xargs} 
\usepackage[textwidth=2.5cm]{todonotes}
\newcommandx{\irina}[2][1=]{\todo[linecolor=red,backgroundcolor=red!25,bordercolor=red,#1]{#2}}
\newcommandx{\cuong}[2][1=]{\todo[linecolor=orange,backgroundcolor=orange!25,bordercolor=orange,#1]{#2}}
\newcommandx{\agnes}[2][1=]{\todo[linecolor=blue,backgroundcolor=blue!25,bordercolor=blue,#1]{#2}}
\newcommandx{\paul}[2][1=]{\todo[linecolor=green,backgroundcolor=green!25,bordercolor=green,#1]{#2}}
\newcommandx{\vesna}[2][1=]{\todo[linecolor=magenta,backgroundcolor=magenta!25,bordercolor=magenta,#1]{#2}}
\newcommandx{\vesnainline}[2][1=]{\todo[linecolor=magenta,backgroundcolor=magenta!25,bordercolor=magenta,#1,inline]{#2}}
\newcommandx{\agnesinline}[2][1=]{\todo[linecolor=blue,backgroundcolor=blue!25,bordercolor=blue,#1,inline]{#2}}

\newcommandx{\tagnes}[2][1=]{\todo[linecolor=green,backgroundcolor=green!25,bordercolor=green,#1]{#2}}

\makeatletter
\newcommand{\mylabel}[2]{#2\def\@currentlabel{#2}\label{#1}}
\makeatother

\newcommand{\WG}[1]{\WW[\![ #1 ]\!]}

\newcommand{\ZG}[1]{\Z_{2}[\![ #1 ]\!]}
\newcommand{\WSG}[1]{\WW[ #1 ]}

\newcommand{\ba}{\overline{\alpha}}

\newcommand{\sD}{\mathcal{D}}

\newcommand{\Sn}{\mathbb{S}}

\newcommand{\tr}{\mathrm{tr}}
\newcommand{\Gal}{\mathrm{Gal}}

\newcommand{\thetazero}{\widetilde{\theta}_\phi}
\newcommand{\thetaclass}{\theta}
\newcommand{\thetaphi}{\theta_\phi}

\newcommand{\Jideal}{\mathcal{J}}
\newcommand{\Iideal}{\mathcal{I}}

\title[Duality resolution at $p=2$]{The duality resolution at $n=p=2$}
\date{\today}

\author[Beaudry]{Agn\`es Beaudry}
\address{Department of Mathematics, University of Colorado Boulder, Campus Box 395, Boulder, CO, 80309, USA}
\email{agnes.beaudry@colorado.edu}

\author[Bobkova]{Irina Bobkova}
\address{Department of Mathematics, Texas A\&M University, College Station, TX, 77843, USA}
\email{ibobkova@tamu.edu}

\author[Henn]{Hans-Werner Henn}
\address{Institut de Recherche Math\'ematique Avanc\'ee, C.N.R.S. et 
Universit\'e de Strasbourg, rue Ren\'e Descartes, 67084 Strasbourg Cedex, France}
\email{henn@math.unistra.fr}

\begin{document}
\maketitle

\begin{abstract}
Working at the prime $2$ and chromatic height $2$, we construct a finite resolution of the homotopy fixed points of Morava $E$-theory with respect to the subgroup $\mathbb{G}_2^1$ of the Morava stabilizer group.
This is an upgrade of the finite resolution of the homotopy fixed points of $E$-theory with respect to the subgroup $\mathbb{S}_2^1$ constructed in work of Goerss--Henn--Mahowald--Rezk, Beaudry and Bobkova--Goerss.
\end{abstract}
\setcounter{tocdepth}{2}
\tableofcontents

 % !TEX root = duality-master.tex

\section{Introduction}
Throughout this paper, we will let $E = E_2$, the Morava $E$-theory spectrum at the prime $2$ over $\F_4$, modeled using the formal group law $\Gamma$, which is either the Honda formal group law $\Gamma_H$ or the formal group law of an elliptic curve $\Gamma_E$. We let $\Sn_2 = \Sn(\Gamma)$ be the group of automorphisms of the formal group law $\Gamma$ over $\F_4$ and $\GG_2= \GG_2(\Gamma)$ its extension by the Galois group $\Gal = \Gal(\F_4/\F_2)$. We let $\Sn_2^1$ and $\G_2^1$ be their norm one subgroups (defined below). We refer the reader to \cite{BarthelBeaudry}, \cite{ghmr}, \cite{BobkovaGoerss} or \cite{HennCent}, among others, for more background on chromatic homotopy theory.

In this paper, we will be discussion resolutions of spectra in the sense of \cite[3.3.1]{HennRes} (see \cref{defn:resofspectra} below).
In  \cite{BobkovaGoerss}, the authors constructed a resolution of spectra
\begin{equation}\label{eq:topres}
\xymatrix{ E^{h\mathbb{S}_2^1}  \ar[r] &  E^{hG_{24}} \ar[r] & E^{hC_6} \ar[r] &  E^{hC_6}\ar[r] & \Sigma^{48}E^{hG_{24}}  }
\end{equation} 
where $G_{24}$ and $C_6$ are certain finite subgroups of $\Sn_2^1$.
This resolution has been central to computations in $K(2)$-local homotopy theory at the prime $2$, for example, in the developments on the Chromatic Splitting Conjecture \cite{BeaudryCSC}, \cite{BGH} and the computation of the exotic Picard group \cite{BBGHPS_Pic}. 
It is an analogue of the duality resolution of $E^{h\GG_2^1}$ at the prime $p=3$ constructed in \cite{ghmr}. 

Despite the fact that the resolution \eqref{eq:topres} has more than proved its worth, the fact that it resolves $E^{h\mathbb{S}_2^1} $ rather than $E^{h\mathbb{G}_2^1} $, is unsatisfactory and inconvenient for computations.  The goal of this paper is to remedy this shortcoming. We state our main result, referring the reader to \cref{sec:finitesubgroups} for the definition of the various groups appearing in this statement:

\begin{theorem}\label{thmintro:topres}
There is a resolution of spectra 
\[
\xymatrix{ E^{h\mathbb{G}_2^1}  \ar[r] &  E^{hG_{48}} \ar[r] & E^{hG_{12}} \ar[r] &  E^{hG_{12}}\ar[r] & \Sigma^{48}E^{hG_{48}}  }.
\]
We call this resolution the \emph{topological duality resolution}.
\end{theorem}

\begin{rem}
It is unfortunate that we only resolve ``half'' of the $K(2)$-local sphere rather than $L_{K(2)}S^0 \simeq E^{h\GG_2}$  itself. This, however, cannot be remedied. The duality resolution at $p=2$ cannot be doubled up as in \cite{ghmr} to give a resolution of $E^{h\mathbb{G}_2}$. To resolve $E^{h\mathbb{S}_2}$ or $E^{h\mathbb{G}_2}$ at $p=2$, one needs larger resolutions such as Henn's centralizer resolution \cite{HennCent}. We do not address this in this paper.
\end{rem}

To obtain the resolution \eqref{eq:topres}, one proceeds as follows. Note that
 applying $E_*(-)$ to \eqref{eq:topres} gives rise to an exact sequence of Morava modules
\begin{equation}\label{eq:mormodres}
\xymatrix@C=1pc{0 \ar[r] & E_*E^{h\SS_2^1}  \ar[r] &   E_*E^{hG_{24}}  \ar[r] &E_*E^{hC_6} \ar[r] & E_*E^{hC_6} \ar[r] & E_*\Sigma^{48}E^{hG_{24}}  \ar[r] & 0. }
\end{equation} 
The first step to building \eqref{eq:topres} is in fact to produce this exact sequence of Morava modules via algebraic means. 
Specifically, in \cite{BeaudryRes}, based on work of  Goerss--Henn--Mahowald--Rezk, the author  constructed an exact sequence of $\Z_2[\![\Sn_2^1]\!]$-modules
\begin{equation}\label{eq:algres}
\xymatrix@C=1pc{ 0 \ar[r] & \Z_2[\![\Sn_2^1/G_{24}']\!] \ar[r] &  \Z_2[\![\Sn_2^1/C_6]\!]  \ar[r] &  \Z_2[\![\Sn_2^1/C_6]\!]  \ar[r] &  \Z_2[\![\Sn_2^1/G_{24}]\!]  \ar[r] & \Z_2 \ar[r] & 0.}
\end{equation}
Applying $\Hom^c_{\Z_2}(-, E_*)$ to the resolution \eqref{eq:algres} produces the exact sequence \eqref{eq:mormodres} of Morava modules. The remainder of the construction consists of proving that it can be realized topologically as in \eqref{eq:topres}. This was a deceivingly difficult task for the resolution of $E^{h\mathbb{S}_2^1}$, achieved in \cite{BobkovaGoerss}. However, in our case, this step goes through exactly like in that paper and so we were saved a lot of pain.

In order to prove \cref{thmintro:topres}, we need to modify this procedure to keep track of the Galois group $\Gal = \Gal(\F_4/\F_2) = \mathbb{G}_2^1/{\mathbb{S}_2^1}$. Note that it is not possible to simply take $\Gal$ fixed points of the resolution \eqref{eq:topres} because the maps are not $\Gal$-equivariant. We need to construct the resolution of \cref{thmintro:topres} from scratch. We  start by constructing the algebraic version of the resolution. Our main algebraic result is as follows.
\begin{theorem}\label{thmintro:algres}
There is an exact sequence in the category of complete left Galois-twisted $\GG_2$-modules (called the \emph{algebraic duality resolution})
\[
\xymatrix@C=1pc{ 0 \ar[r] & \W[\![\GG_2^1/G_{48}']\!] \ar[r] &  \W[\![\GG_2/G_{12}]\!]  \ar[r] &  \W[\![\GG_2/G_{12}]\!] \ar[r] &  \W[\![\GG_2^1/G_{48}]\!]  \ar[r] & \W \ar[r] & 0.}
\]
The maps in the resolution are described in \cref{thm:updatedualres}.
\end{theorem}
Once we have the algebraic resolution, the process of realizing it topologically follows an arc similar to that described for the resolution \eqref{eq:topres}.

\subsection*{Organization of the paper}
In \cref{sec:background}, we review the background necessary for our construction. In \cref{sec:algres}, we construct the algebraic resolution and prove \cref{thmintro:algres}. In \cref{sec:topres}, we realize the resolution  topologically, proving  \cref{thmintro:topres}.

\subsection*{Acknowledgements}
This material is based upon work supported by 
the National Science Foundation under grants No.  DMS-2143811 and DMS-2239362. We thank Mark Behrens, Paul Goerss, Vesna Stojanoska and Viet-Cuong Pham for many conversations related to this work.

 % !TEX root = duality-master.tex

\section{Background}\label{sec:background}
We begin by introducing some of the key objects. We mainly use the notational conventions of \cite{HennCent}. Although this material can be found elsewhere (for example, \cite{BeaudryRes}, \cite{HennCent}, \cite{BeaudryTowards}, \cite{BobkovaGoerss}, \cite{BBGHPS_Pic}), this is all quite technical so we  repeat it here for the convenience of the reader.
\subsection{The Witt vectors}
We let $\W=W(\F_4)$ be the ring of Witt vectors on $\F_4$. Note that $\Gal = \Gal(\F_4/\F_2)$ acts on $\W$. We denote this action by $g(w)= w^{g}$ for $w\in\W$. We fix a choice of primitive third root of unity $\omega \in \W$. We let $\sigma \in \Gal$ be the Frobenius morphism.

\begin{definition}Let $\alpha \in \W$ be given by
\begin{align}\label{eq:alpha}\alpha : = \frac{1-2\omega}{\sqrt{-7}}\end{align}
Here, $\sqrt{-7}$ is the unique element of $\W$ which squares to $-7$  and satisfies $\sqrt{-7} \equiv 1+4 \mod 8$. Such an element exists by Hensel's Lemma. 
We also define
\begin{align}\label{eq:pi} \pi:= 1+2\omega .\end{align}
\end{definition}
\begin{rem}
The element $\alpha$ satisfies the identity
\[\alpha \alpha^{\sigma} =-1\]
and therefore,
$\alpha^{\sigma} = -\alpha^{-1}$ in $\W$. 
Note further that $\pi^2=-3$ and 
\[\pi \pi^{\sigma} =3 .\]
\end{rem}

\subsection{Formal group laws and their endomorphisms} 
In this paper, we consider two choices of formal group laws defined over $\F_2$: The Honda formal group law $\Gamma_H$ and the formal group law of the super-singular elliptic curve with affine equation $y^2+y =x^3$, which we denote by $\Gamma_E$. 

Let $\Gamma=\Gamma_H$ or $\Gamma_E$.
All endomorphisms of $\Gamma$ over the algebraic closure of $\F_4$ are already defined over $\F_4$. In fact, if we let $\xi_{\Gamma}(x)=x^2$, there is an explicit isomorphism
\[ \End_{\F_4}(\Gamma) \cong \W\langle \xi_{\Gamma} \rangle /(\xi_\Gamma w -w^{\sigma}\xi_\Gamma, \xi_\Gamma^2-2u_{\Gamma})\]
where $u_{\Gamma_H}=1$ and $u_{\Gamma_{E}}=-1$. 
Therefore, an element of $ \End_{\F_4}(\Gamma) $ can be expressed uniquely as $a+b\xi_{\Gamma}$ for $a,b \in \W$. 
We define 
\[\mathbb{S}_2(\Gamma) :=  \End_{\F_4}(\Gamma)^{\times} =  \Aut_{\F_4}(\Gamma).\]
The group $\Gal=\Gal(\FF_4/\FF_2)$ acts on $\End_{\F_4}(\Gamma) $ via its action on $\W$. Let
\[\mathbb{G}_2(\Gamma) :=   \mathbb{S}_2(\Gamma) \rtimes \Gal.\]

\begin{rem}
Alternatively, noting that conjugation by $\xi_{\Gamma}$ on $\mathbb{S}_2(\Gamma)$ realizes the action of $\Gal$, we obtain an isomorphism
\begin{equation}\label{eq:G2quotient}
\mathbb{G}_2(\Gamma) \cong \mathbb{D}_2^{\times}(\Gamma)/\langle \xi_{\Gamma}^2\rangle 
\end{equation}
where $\mathbb{D}_2(\Gamma) = \End_{\F_4}(\Gamma)\otimes \Q$.
This is explained in details in \cite{HennCent}.
\end{rem}

We have the following filtration on $\mathbb{S}_2$. For $i\geq 0$ an integer, we let
\begin{align}\label{eq:filt}
F_{i/2}:= F_{i/2}\mathbb{S}_2(\Gamma):= \{g\in \mathbb{S}_2(\Gamma) \mid g\equiv 1 \mod (\xi_{\Gamma}^i)\}.
\end{align}
For $i>0$, the isomorphism maps the equivalence class of $1+a\xi_{\Gamma}^{i}$ for $a\in \W$ to the modulo $2$ reduction of $a$.

We denote $S_2(\Gamma) := F_{1/2} \mathbb{S}_2(\Gamma)$. 
The exact sequence of groups
\[ 1 \to S_2(\Gamma) \to \mathbb{S}_2(\Gamma) \to \F_4^{\times} \cong C_3\to 1\]
splits by sending a generator of $C_3$ to a primitive third root of unity $\omega \in \W^{\times}\subseteq  \mathbb{S}_2(\Gamma)$.

We also define
\[ K(\Gamma) = \overline{\langle \alpha, F_{3/2}\mathbb{S}_2(\Gamma)\rangle},\]
that is, the closure of the subgroup of $\mathbb{S}_2(\Gamma)$ generated by $\alpha$ and $F_{3/2}\mathbb{S}_2(\Gamma)$. 

\begin{lemma}\label{lem:isoS}
The groups $\mathbb{S}_2(\Gamma_H)$ and $\mathbb{S}_2(\Gamma_E)$ are isomorphic and an isomorphism is given by the map
\[ a+b \xi_{\Gamma_E} \mapsto a+b \alpha \xi_{\Gamma_H}.\]
This isomorphism respects the filtration of \eqref{eq:filt} and induces an isomorphism
\[ K(\Gamma_E) \xrightarrow{\cong} K(\Gamma_H).\]
\end{lemma}
\begin{proof}
The first claim is \cite[Lemma 3.1.2]{BeaudryTowards}. Since $\alpha \equiv 1 \mod \xi_H^2$, the map induces an isomorphism $F_{i/2}\mathbb{S}_2(\Gamma_E) \xrightarrow{\cong} F_{i/2}\mathbb{S}_2(\Gamma_H)$ for each $i\geq 0$. Furthermore, it maps $\alpha$ to itself, which then proves the last claim about the subgroups $K(\Gamma)$. 
\end{proof}

\begin{definition}
The group $\G_2(\Gamma)$ contains a central subgroup of order $2$ generated by the inversion $[-1](x)$ of the formal group law $\Gamma$. We denote this subgroup by $C_2 \subseteq \G_2(\Gamma)$. 
For any subgroup $G$ of $\G_2(\Gamma)$, we let 
\[PG = G/(C_2 \cap G).\]
\end{definition}

\begin{lemma}
There is an isomorphism
\[ F_{2/2}\mathbb{S}_2(\Gamma) \cong C_2 \times K(\Gamma).\]
In particular, $PF_{2/2}\mathbb{S}_2(\Gamma) \cong K(\Gamma)$.
\end{lemma}
\begin{proof}
Certainly, $C_2 \times K(\Gamma) \subseteq F_{2/2}\mathbb{S}_2(\Gamma) $ since $-1=1+2\mod \xi_\Gamma^3$ and $\alpha \equiv 1+2\omega \mod \xi_{\Gamma}^3$. 
There is an exact sequence
\[ 1\to F_{3/2}\mathbb{S}_2(\Gamma)  \to F_{2/2}\mathbb{S}_2(\Gamma) \to \F_4  \to 1 \]
and by the identities above, the images of $-1$ and $\alpha$ generate the quotient. Therefore, the containment is an equality. 
\end{proof}

\begin{warn}
The isomorphism of \Cref{lem:isoS} is not equivariant for the action of $\Gal$ and, in fact, 
\[\G_2(\Gamma_H) \not\cong \G_2(\Gamma_E).\]
\end{warn}

In view of \Cref{lem:isoS}, we will often write $\mathbb{S}_2 =\mathbb{S}_2(\Gamma)$ and $K=K(\Gamma)$ for $\Gamma$ either one of $\Gamma_H$ and $\Gamma_E$.

\begin{rem}
The quotient map $\phi \colon \G_2(\Gamma) \to \Gal$ factors as
\[\G_2(\Gamma) \to P\G_2(\Gamma) \to \Gal \to 1\] 
and
 \[ P\G_2(\Gamma) \cong P\mathbb{S}_2(\Gamma) \rtimes \Gal\]
 \end{rem}

\begin{definition}\label{def:G21}
Let $\det \colon \mathbb{S}_2(\Gamma) \to \Z_2^{\times}$ be defined by 
\[\det(a+b\xi_\Gamma) = aa^{\sigma}-2u_{\Gamma}bb^{\sigma}\]
where $a,b \in \W$. We call this map the \emph{determinant}. 
The composition of $\det$ with the quotient $\Z_2^{\times} \rightarrow \Z_2^{\times}/(\pm 1)$ is a map
\[N \colon \mathbb{S}_2 \to \Z_2^{\times}/(\pm 1)\]
whose kernel is denoted by $\mathbb{S}_2^1(\Gamma)$.
Similarly, $\G_2^1(\Gamma)$ is defined to be the kernel of the composite
\[\xymatrix{ \G_2(\Gamma)   \cong \mathbb{S}_2(\Gamma) \rtimes \Gal \ar[rr]^-{\det \times \id} & & \Z_2^{\times} \times \Gal \to \Z_2^{\times}/(\pm 1). }\]
\end{definition}
\begin{rem}
The element $\pi$ satisfies $\det(\pi)=3$. Since $3$ is a topological generator of $\Z_2^{\times}/(\pm 1)$, the map which sends $3 \in \Z_2^{\times}/(\pm 1)$ to $\pi \in \mathbb{G}_2(\Gamma)$ determines a splitting. This gives an isomorphism
\[ \mathbb{G}_2(\Gamma) \cong \mathbb{G}_2^1(\Gamma) \rtimes  \Z_2^{\times}/(\pm 1). \]
There are also isomorphisms 
\[\G_2^1(\Gamma) \cong \mathbb{S}_2^1(\Gamma) \rtimes \Gal, \ \ \ P\G_2^1(\Gamma) \cong P\mathbb{S}_2^1(\Gamma) \rtimes \Gal.\]
Furthermore, the determinant commutes with the isomorphism of \Cref{lem:isoS}. Therefore, $\mathbb{S}_2^1(\Gamma_E) \cong \mathbb{S}_2^1(\Gamma_H)$.
\end{rem}

\begin{definition}\label{def:bigsubgroups}
For any $G \subseteq \G_2(\Gamma)$,  we define the following subgroups.
$G^1  = \G_2^1(\Gamma) \cap  G$,
and $G^0 = \mathbb{S}_2(\Gamma) \cap G$.
 Similarly, for $G \subseteq P\G_2(\Gamma)$, let
 $G^1  = P\G_2^1(\Gamma) \cap  G$ and
 $ G^0 := G \cap  P\mathbb{S}_2(\Gamma)$.
\end{definition}

\begin{rem}\label{rem:isoK}
Let $\mathbb{S}_2 =\mathbb{S}_2(\Gamma)$ and $K=K(\Gamma)$ for $\Gamma = \Gamma_H$ or $\Gamma_E$.
Note that $C_2 \cap K$ is trivial, and so $PK \cong K$. Furthermore, 
although $K$ is \emph{not} Galois invariant in $\mathbb{S}_2$, $PK$ is Galois invariant in $P\mathbb{S}_2$. 
We get isomorphisms
\[K \xrightarrow{\cong} PK \xrightarrow{\cong} PF_{2/2}\mathbb{S}_2.\]
\end{rem}

\subsection{Finite subgroups}\label{sec:finitesubgroups}
Since $\Z_2^{\times}/(\pm 1)$ is free, $\G_2^1(\Gamma)$ contains all the torsion of $ \G_2(\Gamma) $. The finite subgroups of $\G_2(\Gamma)$ are discussed at length in Section 2 of \cite{HennCent}.

There are elements $i,j,k $ of $ \mathbb{S}_2(\Gamma)$ that have order 4 and generate a quaternion subgroup $Q_8$. Explicit choices for these elements are given by
\begin{align*}
i&=\pi^{-1}(1-\varepsilon \xi_{\Gamma}) & j&=\pi^{-1}(1- \omega^2\varepsilon\xi_{\Gamma})& k&=\pi^{-1}(1- \omega\varepsilon\xi_{\Gamma})
\end{align*}
where 
\[\varepsilon = \begin{cases}\alpha & \Gamma=\Gamma_{H} \\
1 & \Gamma=\Gamma_{E} .
\end{cases}\]
The $i,j,k$ satisfy the classical quaternion relations with $ij=k$.
We note the relation
\[\omega =- \frac{1}{2}(1+i+j+k).\]
Furthermore,
\[\omega i \omega^{-1}=j, \quad \omega j\omega^{-1}=k, \quad  \omega k\omega^{-1} =i.\] 
Together, these elements thus form a binary tetrahedral group of order $24$
\[G_{24} = \{\pm 1, \pm i , \pm j, \pm k, \frac{1}{2}( \pm1\pm i \pm j\pm k)\} = \langle i, \omega \rangle.\]

As in Lemma 2.2 of \cite{HennCent}, conjugation by the element $1+i$ in $\mathbb{D}_2(\Gamma)^{\times}$ stabilizes $G_{24}$. Letting $[1+i]$ denote the image of $1+i$ in $\GG_2(\Gamma)$ (expressed as a quotient as in \eqref{eq:G2quotient}), the subgroup of $\mathbb{G}_2(\Gamma)$ generated by $G_{24}$ and $[1+i]$ is a choice of maximal finite subgroup $G_{48}(\Gamma)$. That is,
\[G_{48} = G_{48}(\Gamma) = \langle i, \omega, [1+i] \rangle \subseteq \GG_2(\Gamma).\]

We have the following unique conjugacy class of maximal finite subgroups of $ \G_2(\Gamma)$. We let
\begin{align*}
G_{48}(\Gamma) \cong \begin{cases}  GL_2(\F_3), & \Gamma = \Gamma_E \\
O_{48}, & \Gamma = \Gamma_H
\end{cases}
\quad \subseteq \G_2(\Gamma).
\end{align*}
where $O_{48}$ denotes the binary octahedral group.
When $\Gamma$ is understood from context or for results that are independent of the choice of formal group law, we simply write
\[G_{48}=G_{48}(\Gamma).\]
The intersection of $G_{48}(\Gamma)$ with $ \mathbb{S}_2(\Gamma)$ is also maximal and, for both $\Gamma_E$ and $\Gamma_H$, there is an isomorphism
\[G_{24}:= G_{48}(\Gamma) \cap \mathbb{S}_2(\Gamma) \cong Q_8 \rtimes C_3,\]
where $Q_8$ is the quaternion group and $C_3 \cong \F_4^{\times}$ is generated by our choice of a primitive third root of unity $\omega \in\W^{\times} \subseteq \mathbb{S}_2(\Gamma) $.

The subgroup of $G_{24}$ generated by $i$ is cyclic of order $4$ and so we simply denote it by $C_4$. Furthermore, 
\[C_{8}= C_8(\Gamma) = \langle i, [1+i] \rangle \subseteq G_{48}(\Gamma) \subseteq \mathbb{G}_2(\Gamma) \cong \mathbb{D}_2^{\times}/(\xi_{\Gamma}^2)\]
is cyclic of order $8$ since the identity $(1+i)^2=2i$ in $\mathbb{D}_2^\times$ implies in $\G_2(\Gamma)$ that
\begin{align*}
[1+i]^2 =\begin{cases} i & \Gamma= \Gamma_H \\
-i & \Gamma =\Gamma_E.
\end{cases}
\end{align*}

We will also use the following notation:
\begin{align*}
G_{12}=G_{12}(\Gamma) = \langle -\omega, [j-k] \rangle \cong \begin{cases}  C_2 \times \mathfrak{S}_3 & \Gamma = \Gamma_E \\
C_3 \rtimes C_4 & \Gamma = \Gamma_H
\end{cases}
\quad \subseteq G_{48}(\Gamma)
\end{align*}
where $ \mathfrak{S}_3$ is the symmetric group on 3 letters.
Here, $[j-k]$ is the image of $j-k$ in $\GG_2(\Gamma)$. That this is a subgroup of $G_{48}(\Gamma)$ follows from the relation 
\[ [j-k] =j[1+i].\] 
Note that,
\begin{equation}\label{eq:jmk}  j-k =\begin{cases} \alpha \xi_{\Gamma_H} & \Gamma =\Gamma_H \\ 
 \xi_{\Gamma_E} & \Gamma =\Gamma_E 
 \end{cases}
\end{equation}
in $\mathbb{D}_2^\times$,  so conjugation by $[j-k]$ induces the Galois action on $\W$.

The inclusion follows from the identity $j(j-k)= - (1+i)$ in $\mathbb{D}_2^{\times}$.
For both $\Gamma_E$ and $\Gamma_H$, we have
\[ G_{12}(\Gamma) \cap \mathbb{S}_2(\Gamma) \cong C_2 \times C_3 \cong C_6.\]

Finally, we let
\begin{align*}
G_{48}'&=G_{48}'(\Gamma)= \pi G_{48}'(\Gamma) \pi^{-1}\\
G_{24}'&=\pi G_{24} \pi^{-1}.
\end{align*}

\begin{warn}
The groups $G_{48}(\Gamma_H)$ and $G_{48}(\Gamma_E)$ are not isomorphic. Nor are the groups $G_{12}(\Gamma_H)$ and $G_{12}(\Gamma_E)$.
\end{warn}

\begin{rem}\label{rem:symmetricgroup}
In Lemma 2.2 of \cite{HennCent}, it is shown that there is an isomorphism
\[ PG_{48}(\Gamma) \cong \mathfrak{S}_4  ,\]
the symmetric group on four letters, independent of the choice of $\Gamma = \Gamma_E$ or $\Gamma_H$. 
 Similarly, 
\[PG_{12}(\Gamma) \cong \mathfrak{S}_3\]
independent of $\Gamma$.
\end{rem}

\subsection{Twisted modules and group rings}
We will be studying group rings and their modules with coefficients in $\W$. To avoid ambiguities, we will use $\zeta $ to denote a primitive third root of unity in the coefficients $\W$ and $\omega$ to denote the element of $\G_2(\Gamma)$, so that for the coefficients, we have
\[\W \cong \Z_2[\zeta]/(1+\zeta+\zeta^2).\]

The group $\G_2(\Gamma)$ acts on $\W$ via the quotient map
\[\phi \colon \G_2(\Gamma) \to  \Gal.\]
We write $w^g = g(w)$ for $w \in \W$. Restriction along the inclusion of any closed subgroup $G \subseteq \G_2(\Gamma)$ also gives $\W$ the structure of a $G$-module. Note that $G^0$ as defined in \Cref{def:bigsubgroups} is the kernel of $\phi$ restricted to $G$ and that $G^0$ acts trivially on $\W$. Furthermore, the factorization of $\phi$  
\[\xymatrix{\G_2 \ar[r]^-{\phi} \ar[d] & \Gal \\
P\G_2 \ar[ur] & }
\]
 also gives $\W$ the structure of a $G$-module for any subgroup $G$ of $P\G_2$.

\begin{definition}
Let $G$ be a closed subgroup of $\G_2(\Gamma)$. 
The \emph{Galois-twisted group ring} $\WG{G}$ is the profinite abelian group 
 \[\WG{G} = \lim_{i,j} (\W/p^j)_{\phi}[G_i]  \] 
  endowed with
the multiplicative structure determined by
\[(\alpha_1 g_1) (\alpha_2 g_2) = \alpha_1 \alpha_2^{g_1} g_1 g_2 , \ \ \  g_1,g_2\in G, \ \alpha_1, \alpha_2\in \W\]
where $G$ acts on $\W$ via $\phi \colon G \rightarrow \Gal$. 

A \emph{Galois-twisted} $G$-module is a profinite $\WG{G}$-module $M$ with the property that $g(\alpha m) = \alpha^{\phi(g)}g(m)$ for any $m\in M$, $\alpha \in \W$ and $g\in G$. Let $\Mod^\phi(G)$ denote the category of Galois-twisted $G$-modules, with morphism the equivariant maps.
\end{definition}

\begin{definition}\label{def:completedGset}
Let $G$ be a closed subgroup of $\G_2(\Gamma)$ or of $P\G_2(\Gamma)$.
For $X = \lim_i X_i$ a profinite $G$-set, the \emph{Galois-twisted completed $G$-module} associated to $X$ is defined as
\[\WG{X} = \lim_{i,j} (\W/p^j)[X_i]  \] 
where the $G$-module structure on $(\W/p^j)[X_i]$ satisfies $g(wx) =w^g (gx) $ for all $w\in \W/p^j$ and $x\in X_i$. 
The \emph{augmentation} is the map  
\[\varepsilon \colon \WG{X} \to \W\]
obtained from the maps
\[ \varepsilon_i \colon (\W/p^j)[X_i]  \to  (\W/p^j)  \]
that send elements $x_i \in X_i$ to $1$. We let $ I X$ denote the kernel of the augmentation.
\end{definition}

\begin{rem}
It will be useful to note that 
\[\WG{X} \cong \W \otimes_{\Z_2} \Z_2[\![X]\!] \]
where the action on the right is the diagonal action.
\end{rem}

See \cite{HennCent} for a more detailed discussion of these constructions.

\begin{definition}\label{defn:uparrownotation}
Let $G$ be a subgroup of $\GG_2$ or $P\G_2$ and $H$ be a finite subgroup of $G$. We let
\[\W {\uparrow}_{H}^{G} = \WG{G/H} \cong \WG{G} \otimes_{\W[ H ]} \W \]
where the action of $G$ and $H$ on $\W$ is Galois-twisted. Similarly, we let
\[\Z_2 {\uparrow}_{H}^{G} = \Z_2[\![G/H]\!] \cong \Z_2[\![G]\!] \otimes_{\Z_2[ H ]} \Z_2 \]
\end{definition}

\begin{lemma}\label{lem:connecting}
Let $G \subseteq \GG_2$ or $G \subseteq P\GG_2$ be a closed subgroup.
The exact sequence of $\WG{G}$-modules
\[ 0 \to IG \to \WG{G} \to \W \to 0\]
gives rise to a long exact sequence on continuous group homology whose connecting homomorphism
\[\delta \colon H_{n} (G,  \W ) \to H_{n-1}(G, IG) \] 
for $n\geq 1$ is an isomorphism. 
\end{lemma}
\begin{proof}
This is a consequence of the fact that
\[H_*(G, \WG{G}) \cong \begin{cases}  0 & *>0 \\
\W & G=G^0\\
\Z_2 & G\neq G^0.
\end{cases}\]
Indeed, since $G^0$ acts trivially on $\W$, $H_*(G^0, \W[\![G]\!]) \cong \W$ concentrated in homological degree zero. If $G=G^0$, this proves the claim, otherwise, note that $\W \cong \Z_2[\Gal]$ as a $\Gal$-module, and then the result follows easily using the Lyndon--Hochschild--Serre for the group extension $G^0 \subseteq G$. 
\end{proof}

\begin{rem}\label{rem:connecting}
Suppose that $G=G^0$. The connecting homomorphism  $\delta$ of \Cref{lem:connecting} when $n=1$ can be described as follows. Since $G$ acts trivially on $\W$ and $\W$ is of finite rank over $\Z_2$, we can write $\delta=\delta' \otimes_{\Z_2}\W$ where $\delta'$ is the connecting homomorphism in homology corresponding to the exact sequence 
\[ 0 \to I_{\Z_2}G \to \ZG{G} \to \Z_2 \to 0.\]
Identify
\[H_1(G,  \Z_2) \cong G/\overline{[G,G]}, \ \ \ \ \ \ H_0(G, I G) \cong  I_{\Z_2} G/\overline{(I_{\Z_2}G)^2}\]
and let $\overline{g}$ denote its residue class in $G/\overline{[G,G]} $.
Then, under these identifications, 
\[\delta'(\overline{g}) = e-g  \mod \overline{(I_{\Z_2}G)^2} .\]
Here, $\overline{[G,G]} $ and $\overline{(I_{\Z_2}G)^2}$ denote the closures of the corresponding subgroups. See \cite[Lemma A.1.4]{BeaudryRes} for more details.
\end{rem}

\subsection{The Poincar\'e duality subgroup}
Now, we record some facts about the group $K$ which we will need below. From Proposition 2.5.1 of \cite{BeaudryRes}, we have:
\begin{lemma}
The subgroup $K(\Gamma)$ is normal in $\mathbb{S}_2(\Gamma)$ and $\mathbb{S}_2(\Gamma) \cong K(\Gamma)\rtimes G_{24}$. Similarly, $K^1(\Gamma)$ is normal in $\mathbb{S}_2^1(\Gamma)$ and $\mathbb{S}_2^1(\Gamma) \cong K^1(\Gamma)\rtimes G_{24}$.
\end{lemma}

\begin{rem}\label{warn:KnotGal}
The group $K(\Gamma)$ is \emph{not normal} in $\mathbb{G}_2(\Gamma)$. Although $F_{3/2}\mathbb{S}_2$ is normal, $\alpha^{\sigma} = -\alpha^{-1}$ and $-1\not\in K$. However, the group $PK(\Gamma) \cong F_{2/2}\mathbb{S}_2$ is normal in $P\mathbb{G}_2(\Gamma)$, and so is $PK^1$ in $P\G_2^1$. Furthermore, we have
\[ P\mathbb{G}_2(\Gamma) \cong PK(\Gamma) \rtimes PG_{48}, \ \ \ P\mathbb{G}_2^1(\Gamma) \cong PK^1(\Gamma) \rtimes PG_{48} .\]
\end{rem}

In our arguments below, we will use the following result which is Lemma 2.5.11 of \cite{BeaudryRes} together with the isomorphism $PK \cong K$, the fact that $K$ acts trivially on $\W$ and that $\W$ is of finite rank over $\Z_2$.
In order to state it, we define
\begin{definition}\label{def:alphatau}
For $\tau=i,j,k$ as above and $\alpha$ as in \eqref{eq:alpha}, let
\[\alpha_{\tau} := [\tau, \alpha]= \tau\alpha\tau^{-1} \alpha^{-1}.\] 
\end{definition}

For the next statements, note that there is an exact sequence
\begin{align}\label{eq:PKPK1}
1 \to PK^1 \to PK \to PK/PK^1\cong \Z_2 \to 1
\end{align}
and that the element $\alpha\pi \in F_{4/2}PK$ maps to a generator of the quotient, and so defines a splitting. Note that the element $\pi$ also maps to a generator, but when considering residual actions on homology below, we will find that $\alpha \pi$ is a better choice for the splitting.

We recall some results about the groups $PK$ and $PK^1$.
\begin{lemma}[Cor 2.5.12 of \cite{BeaudryRes}]\label{lem:PD}
The group $PK^1$ is an orientable Poincar\'e duality group of dimension $3$ and $PK$ of dimension $4$. In particular,
\[H^n(PK^1, \WG{PK^1}) \cong  \begin{cases}\W & n=3 \\
0 & \text{otherwise}
\end{cases} \]
and
\[H^n(PK, \WG{PK}) \cong  \begin{cases}\W & n=4 \\
0 & \text{otherwise}.
\end{cases} \]
\end{lemma}

We end with the following lemma, which is a slight improvement on results of \cite{BeaudryRes}. For its statement, recall from \cref{rem:connecting} that $\overline{g}$ denotes the residue class of $g\in G$ in $G/\overline{[G,G]}$.
\begin{lemma}\label{lem:H1KZ2}
There is an isomorphism
\[H_n(PK^1 , \W) \cong \begin{cases}  \W  &  n=0,3\\
 \W/4\{\overline{\alpha} \} \oplus \W/2\{\overline{\alpha}_i, \overline{\alpha}_j\} & n=1 \\
 0 & n=2 \ \text{or} \ n>3.
\end{cases}. \]
The conjugation action of $PQ_8$ on $H_1(PK^1,\W)$
is given by
\begin{align*}
i_*(\overline{\alpha}) &=  \overline{\alpha} + \overline{\alpha}_i & j_*(\overline{\alpha}) &=  \overline{\alpha}+ \overline{\alpha}_j\\
i_*(\overline{\alpha}_i) &=  \overline{\alpha}_i & j_*(\overline{\alpha}_i) &=  \overline{\alpha}_i + 2\overline{\alpha}\\
i_*(\overline{\alpha}_j) &=  \overline{\alpha}_j + 2\overline{\alpha} & j_*(\overline{\alpha}_j) &=  \overline{\alpha}_j,
\end{align*}
In particular, $H_1(PK^1, \W)$ is generated by the image of $\alpha$ as a $\WSG{PG_{48}}$-module.

The conjugation action of $PK/PK^1$ on $H_n(PK^1;\W)$ is trivial and there is an isomorphism
\[H_n(PK , \W) \cong 
H_{n}(PK^1;\W)\oplus H_{n-1}(PK^1;\W).\]
In particular,
\[H_1(PK , \W) \cong  H_1(PK, \Z_2) \otimes_{\Z_2} \W \cong  \W/4\{\overline{\alpha} \} \oplus \W/2\{\overline{\alpha}_i, \overline{\alpha}_j\} \oplus \W\{\overline{\alpha\pi} \}.\]
The  conjugation action of $PQ_8$ on $H_1(PK,\W)$ is trivial on $\overline{\alpha\pi}$ and as in \cref{lem:H1KZ2} on the other generators.
\end{lemma}
\begin{proof}
The claims about $PK^1$ are immediate consequences of \cite{BeaudryRes}. Specifically, Lem. 2.5.11 and Cor. 2.5.14 of  \cite{BeaudryRes}.

That the action of $PK/PK_1$ is trivial on the homology of $PK^1$ is clear for $H_1$ since $\alpha\pi \in F_{4/2}PK$ and the relevant commutators are in too deep a filtration. The action on $H_3(PK^1;\W)$ is also trivial since $PK$ is an orientable Poincar\'e duality group of dimension $4$ (\cite[Cor.  2.5.12]{BeaudryRes}) and a non-trivial action would imply that $H_4(PK;\W) = 0$ when this group must be $\W$. Finally, the Lyndon--Serre--Hochschild spectral sequence for the extension \eqref{eq:PKPK1} collapses at $E_2$ for degree reasons. 
\end{proof}

 % !TEX root = duality-master.tex

\section{Algebraic duality resolution}\label{sec:algres}
In this section, we will let $\Gamma$ be either of $\Gamma_H$ or $\Gamma_E$. We let $\GG_2 =\GG_2(\Gamma)$ and
for any subgroup of $G\subset \GG_2(\Gamma)$, we will let $G=G(\Gamma)$ as well. We will use  \cite{BeaudryRes}. While that paper is focused on $\Gamma= \Gamma_H$, since $\SS_2(\Gamma_H) \cong \SS_2(\Gamma_E)$, all of \cite{BeaudryRes}  can be adapted directly to the case of $\Gamma_E$ and we do this implicitly. Furthermore, the construction in \cite{BeaudryRes} could have been done for $P\SS_2^1$ as well and we use this implicitly as well.  The following result will be proved in this section. It is a consequence of  Theorem 1.2.1 of \cite{BeaudryRes}. We will prove it later in this section.

\subsection{Summary of Results}

Recall that $S_2^1 = S_2 \cap \mathbb{G}_2^1$ where $S_2$ consists of the elements of $\mathbb{S}_2$ that are congruent to $1$ modulo $\xi_\Gamma$. 
Let
\begin{equation}\label{defn:idealJ}
\Jideal=(2, I  K^1 \cdot I  S_2^1,  I  S_2^1 \cdot I  K^1 ). 
\end{equation}
\begin{theorem}\label{thm:classicdualres}
Let $\Z_2$ be the trivial $\Z_2[\![\Sn_{2}^1]\!]$--module. There is an exact sequence of complete left $\Z_2[\![\Sn_{2}^1]\!]$--modules
\begin{align*}
0 \rightarrow   \Z_2{\uparrow}^{\Sn_{2}^1}_{G_{24}'}  \xra{\partial_3}  \Z_2 {\uparrow}^{\Sn_{2}^1}_{C_{6}}  \xra{\partial_2} \Z_2 {\uparrow}^{\Sn_{2}^1}_{C_{6}}  \xra{\partial_1}  \Z_2 {\uparrow}^{\Sn_{2}^1}_{G_{24}} \xra{\varepsilon} \Z_2 \rightarrow 0.
\end{align*}
Let $e$ be the unit in $\Z_2[\![\Sn_{2}^1]\!]$ and $e_0$ be the resulting generator of $  \Z_2 {\uparrow}^{\Sn_{2}^1}_{G_{24}}$, $e_1$ of the first $\Z_2 {\uparrow}^{\Sn_{2}^1}_{C_{6}} $ and so on.
The maps $\partial_p$ can be chosen so that
\[\partial_p(e_p) = \delta_pe_{p-1} \]
for $\delta_1,\delta_2,\delta_3 \in \Z_2[\![\SS_2^1]\!]$ which satisfy
\begin{align*}
\delta_1e_0  &=  (e-\alpha)e_0\\
\delta_2e_1 &\equiv  (\alpha+i+j+k)e_1 \mod \Jideal\\
\delta_3e_2&= \pi(e+i+j+k)(e-\alpha^{-1})\pi^{-1}e_2.
\end{align*}
Similarly, there is a resolution of $\Z_2[\![P\Sn_{2}^1]\!]$
\[0 \rightarrow   \Z_2{\uparrow}^{P\Sn_{2}^1}_{PG_{24}'}  \xra{\partial_3}  \Z_2 {\uparrow}^{P\Sn_{2}^1}_{PC_{6}}  \xra{\partial_2} \Z_2 {\uparrow}^{P\Sn_{2}^1}_{PC_{6}}  \xra{\partial_1}  \Z_2 {\uparrow}^{P\Sn_{2}^1}_{PG_{24}} \xra{\varepsilon} \Z_2 \rightarrow 0,
\]
with the maps described as above.
\end{theorem}

The middle map is quite subtle. It is the composition of two maps, neither of which is explicit. In fact,
the element $\delta_2$ is not unique and there is no precise formula for it.
Rather, it is characterized by various properties, from which it can be approximated. See \cref{sec:middle} and, in particular, \cref{cor:finalbestapproxdelta2}
  below for more details and an approximation of $\delta_2$ modulo the ideal
\begin{equation}\label{defn:idealI}
\Iideal=(4, 2(IPS_2^1)^2, 2IPK^1, (IPK^1)^2, IPK^1 \cdot (IPS_2^1)^2, (IPS_2^1)^2 \cdot IPK^1 ) .
\end{equation}

The goal of this section is to construct a similar resolution for the group $\GG_2^1$ in the category of Galois-twisted modules. Recall the notation of \cref{defn:uparrownotation} for the following statement. For $x\in \WG{\G_2^1}$ (or $ \WG{P\G_2^1}$), we let
\[\tr_{\sigma}(x) = -(\zeta x+\zeta^\sigma x^{\sigma}).\]
The sign is chosen so that for a $\sigma$-invariant $x$ we get $\tr_\sigma(x) = x$.
\begin{theorem}\label{thm:updatedualres}
There is an exact sequence of complete left $\WG{\G_2^1}$--modules
\[0 \to\W {\uparrow}^{\G_{2}^1}_{G_{48}'}  \xra{\partial_3}\W {\uparrow}^{\G_{2}^1}_{G_{12}} \xra{\partial_2}\W {\uparrow}^{\G_{2}^1}_{G_{12}}   \xra{\partial_1}  \W {\uparrow}^{\G_{2}^1}_{G_{48}}  \xra{\varepsilon} \W \to 0 .\]
For the generators named as in \cref{thm:classicdualres}, 
the maps $\partial_p$  can be chosen to satisfy 
\[\partial_p(e_p)  = \delta_p^\phi e_{p-1}\]
for  $\delta_1^\phi, \delta_2^\phi, \delta_3^\phi \in \WG{\GG_2^1}$ which satisfy
\begin{align*}
\delta_1^\phi e_0&= \tr_\sigma(\delta_1) e_0 \\
\delta_2^\phi e_2 &\equiv \tr_\sigma(\delta_2)e_1 \mod \Iideal \\
\delta_3^\phi e_3 &= \tr_\sigma(\delta_3) e_2 
\end{align*}
for $\delta_1,\delta_2, \delta_3$ as in \cref{thm:classicdualres}.
Similarly, there is an exact sequence of complete left $\WG{P\G_2^1}$--modules
\[0 \to  \W {\uparrow}^{P\G_{2}^1}_{PG_{48}'}\xra{\partial_3}  \W {\uparrow}^{P\G_{2}^1}_{PG_{12}}  \xra{\partial_2}  \W {\uparrow}^{P\G_{2}^1}_{PG_{12}}  \xra{\partial_1}  \W {\uparrow}^{P\G_{2}^1}_{PG_{48}} \xra{\varepsilon} \W \to 0 \]
with maps described exactly as above.
\end{theorem}

\begin{warn}
The reader might be tempted to think that there is an easy procedure to transform our $\Z_2[\![\SS_2^1]\!]$-complex into a $\WG{\G_2^1}$-complex by replacing the modules in the obvious way and applying $\tr_{\sigma}$ to the maps, but this is not the case. The fact that we were able to express the maps this way is not formal. 
Again, the middle map $\partial_2$ is the composition of two non-explicit maps and is only approximated by $\tr_{\sigma}(\delta_2)$. See \cref{sec:middle}  for an in depth description of the middle map.
\end{warn}

The rest of this section is devoted to proving \Cref{thm:updatedualres}, which we do by emulating the construction of \cite{BeaudryRes}. This section will follow the structure of Section 3 of that reference. 
There, the trick is to construct a resolution of free $\Z_2[\![K^1]\!]$-modules. However, since $K^1$ is not Galois invariant, in our case, the right approach is to first construct a resolution of  $\WG{P\GG_2^1}$-modules as the subgroup $PK^1$ is Galois invariant.
 However, note again that $PK \cong K$ and $PK^1\cong K^1$ as abstract groups, so many of the results of \cite{BeaudryRes} apply.  If we construct resolution of $\WG{P\GG_2^1}$-modules, then the resolution for  $\WG{\GG_2^1}$ is obtained directly by restriction along the quotient map $\GG_2^1 \rightarrow P\GG_2^1$. So, it is sufficient to construct the resolution of  $\WG{P\GG_2^1}$-modules.  For this reason, for the rest of this section, we work with $P\GG_2^1$.

\begin{rem}
As we construct the resolutions, it will be convenient to use the following notation,
\[0 \to\sD_3 \xra{\partial_3}\sD_2\xra{\partial_2}\sD_1  \xra{\partial_1}\sD_0\xra{\varepsilon} \W \to 0 \]
where
\begin{align*}
\sD_0 &= \W {\uparrow}^{P\G_{2}^1}_{PG_{48}}  &  \sD_1 &=  \W {\uparrow}^{P\G_{2}^1}_{PG_{12}} &
 \sD_2 &=  \W {\uparrow}^{P\G_{2}^1}_{PG_{12}} &  \sD_3  &= \W {\uparrow}^{P\G_{2}^1}_{PG_{48}'} .
\end{align*}
We denote the canonical generator of $\sD_p$ by $e_p$. 
\end{rem}

\begin{warn}
For the purpose of the later sections, although we model our arguments on \cite[Section 3]{BeaudryRes}, some of our notation is designed to make further references to \cite{HennCent} more transparent. For example, in \cite{BeaudryRes}, $N_0$ is the analogue of the module we denote by $M_1$ here.
\end{warn}

\subsection{The construction of the resolution}
The resolution of   \Cref{thm:updatedualres} is spliced from three exact sequences which we construct in \Cref{lem:N0}, \Cref{lem:N1}  and \Cref{lem:M3} which make up the technical core of this paper. 
\begin{lemma}\label{lem:N0}
Let $\sD_0 = \W {\uparrow}^{P\G_{2}^1}_{PG_{48}}$ and let $\varepsilon$ be the augmentation map. Define $M_1$ by the short exact sequence of Galois twisted $\WW[\![P\GG_2^1]\!]$-modules
\begin{equation}\label{eq:M1}
\xymatrix{ 0 \ar[r] & M_1 \ar[r] & \sD_0\ar[r]^-{\varepsilon} & \W \ar[r] \ar[r] & 0.}
\end{equation}
Then $M_1$ is the left $\WG{P\G_2^1}$-submodule of $ \W {\uparrow}^{P\G_{2}^1}_{PG_{48}} $ generated by
\[ \tr_\sigma (\delta_1) e_0= - (\zeta(e-\alpha) + \zeta^{\sigma} (e-\alpha^{\sigma}))e_0\]
for $e$ the unit of $P\GG_2^1$ and $\alpha$ as defined in \eqref{eq:alpha}. Furthermore, $\alpha^{\sigma} e_0= \alpha^{-1}e_0$.
\end{lemma}
\begin{proof}
The composition $PK^1 \to P\GG_2^1 \to P\GG_2^1/PG_{48}$ given by the inclusion followed by the projection onto the coset space is a bijection. Therefore, the short exact sequence defining $M_1$ restricts to a short exact sequence of $\WG{PK^1}$-modules
\[ \xymatrix{ 0 \ar[r] & M_1 \ar[r] &  \WG{PK^1} \ar[r]^-{\varepsilon} & \W \ar[r] \ar[r] & 0.}\]
It follows that, as $\WG{PK^1}$-modules, there is an
isomorphism $M_1 \cong IPK^1$. We will prove that 
\[ H_0( PK^1, M_1) \cong I_{\Z_2}PK^1/(I_{\Z_2}PK^1)^2 \otimes_{\Z_2} \W\] 
is generated by the image of
 \begin{align}\label{eq:theelementforfirstdiff}
-  (\zeta(e-\alpha) + \zeta^{\sigma} (e-\alpha^{\sigma}))e_0
  \end{align}
 as a $\W[PG_{48}]$-module. 
 
 Assuming this, we can finish the proof as follows.
Define $F\colon \W[\![P\G_2^1]\!]   \to M_1$ to be the module map which sends the unit in $\W[\![P\G_2^1]\!]$ to the element \eqref{eq:theelementforfirstdiff} of $M_1$. Then
\[\F_4\otimes_{\W[\![PK^1]\!]} F \colon \F_4\otimes_{\W[\![PK^1]\!]}\W[\![P\G_2^1]\!]  \to \F_4\otimes_{\W[\![PK^1]\!]}M_1  \]
is surjective since the source is $\F_4[ PG_{48}]$ and we will see that the target is cyclic as an $\F_4[ PG_{48}]$-module. Applying the Nakayama type Lemma 4.4 of \cite{HennCent} then proves the surjectivity of $F$, which implies the claim about $M_1$.

So, it remains to understand $H_0( PK^1, M_1) $ as an $\W[PG_{48}]$-module.
Since $PK^1 \cap \Gal =e$, the group $PK^1$ acts trivially on $\W$.
As in \Cref{lem:connecting} and \Cref{rem:connecting}, the long exact sequence on homology groups gives 
an isomorphism 
\[\xymatrix{ 
H_1(PK^1, \W) \ar[d]^-\cong \ar[r]^-{\delta'} & H_0( PK^1, M_1) \ar[d]^-\cong \\
 PK^1/\overline{[PK^1,PK^1 ]} \otimes_{\Z_2} \W \ar[r]   &   I_{\Z_2}PK^1/(I_{\Z_2}PK^1)^2  \otimes_{\Z_2} \W}\]
where the bottom horizontal map sends $\overline{g}$ to $e-g$. Therefore, it suffices to prove that $ \zeta \ba +\zeta^{\sigma} \ba^{\sigma}$ generates $H_1(PK^1, \W) $ as a  $\WSG{ PG_{48}}$-module.

The image of $\alpha^{\sigma}=-\alpha^{-1}$ under the map 
\[p \colon PK^1 \to PK^1/\overline{[PK^1,PK^1 ]} \cong H_1(PK^1, \Z_2)  \] is the same as that of $\alpha^{-1}$ since we are working modulo $C_2=\{\pm 1\}$. The image of $\alpha^{-1}$ is $-\ba$ as $p$ is a group homomorphism. Here, note that the $-1$ is a coefficient and $-\ba$ should not be confused with $\overline{-\alpha} = \overline{\alpha}$.
It follows that
\begin{align*}
 \zeta \ba +\zeta^{\sigma} \ba^{\sigma} &=  (1+2\zeta)\ba     \in \W/4\{\bar \alpha\}.
 \end{align*}
 Since $(1+2\zeta)$ is a unit, it suffices to prove that $\ba$ generates as an $\WSG{ PG_{48}}$-module. But this follows directly from the formulas for the action of $PQ_8$ on $\ba$ of \Cref{lem:H1KZ2}: In fact, $\ba$ is a generator for $H_1(PK^1, \W)$  as a $\WSG{ PQ_8}$-module.
\end{proof}

\begin{rem}\label{rem:taucom}
For any $\tau \in PG_{12}$, the following holds in  $\W[\![P\G_2^1]\!]$:
\[ \tau(\zeta(e-\alpha) + \zeta^{\sigma} (e-\alpha^{\sigma}))= (\zeta(e-\alpha) + \zeta^{\sigma} (e-\alpha^{\sigma}))\tau \]
This can be deduced from the fact that $\tau  \zeta  = \zeta^{\tau} \tau$ and the fact that $\tau \alpha= \alpha^{\tau} \tau$
where here we use that $PC_6$ commutes with $\alpha$ and 
\[[j-k] \alpha = \alpha^{\sigma} [j-k]\]
which follows from \eqref{eq:jmk}.
\end{rem}

\begin{lemma}\label{lem:N1partial}
Let $M_1$ be as in \Cref{lem:N0}. Let $ \sD_1= \W{\uparrow}^{P\G_{2}^1}_{PG_{12}}$ with canonical generator $e_1$. There is a map $\partial_1 \colon \sD_1\to M_1$ of $\WG{P\G_2^1}$-modules determined by
\[\partial_1(e_1) = \tr_\sigma(\delta_1)e_0=  -(\zeta(e-\alpha) + \zeta^{\sigma} (e-\alpha^{\sigma}))e_0 . \]
\end{lemma}

\begin{proof}
The inclusion $G_{12}(\Gamma) \subseteq G_{48}(\Gamma)$ induces an inclusion $PG_{12} \subseteq PG_{48}$.
So, for $\tau \in PG_{12}$, $\tau e_0 = e_0$.
  On the one hand, 
\[\partial_1(\tau e_1) = \partial_1(e_1)=-(\zeta(e-\alpha) + \zeta^{\sigma} (e-\alpha^{\sigma}))e_0. \] 
On the other, 
\begin{align*}
\tau \partial_1( e_1) 
&=-\tau (\zeta(e-\alpha) + \zeta^{\sigma} (e-\alpha^{\sigma}))e_0 \\
& =-   (\zeta(e-\alpha) + \zeta^{\sigma} (e-\alpha^{\sigma})) \tau e_0 \\
& = -  (\zeta(e-\alpha) + \zeta^{\sigma} (e-\alpha^{\sigma}))  e_0 
\end{align*}
where we used \cref{rem:taucom}.
This shows that $\partial_1$ is well defined. 
\end{proof}

\begin{lemma}\label{lem:N1}
Let $M_2$ be defined by the short exact sequence
\begin{equation}\label{eq:M2}
\xymatrix{0 \ar[r] &  M_2 \ar[r] &  \sD_1\ar[r]^-{\partial_1} &  M_1 \ar[r] & 0 }.
\end{equation}
There exists $\thetazero  \in \WG{P\G_2^1}$ such that
\begin{enumerate}
\item  $\thetazero  e_1$ is in the kernel of $\partial_1$, and
\item  $\thetazero  e_1 \equiv (3+i+j+k)e_1 \mod (4,  IPK^1)$.
\end{enumerate}
If $\thetazero $ satisfies these conditions, then  $\thetazero  e_1$ generates $M_2$ as a $ \WG{P\G_2^1}$-module.
\end{lemma}
\begin{proof}

By  \Cref{lem:N0}, $\partial_1$ surjects onto $M_1$. Let $M_2 = \ker(\partial_1)$. Note that
there is
an isomorphism of $\W[PG_{48}]$-modules
\[H_0(PK^1,  \sD_1) \cong   \WSG{PG_{48}/PG_{12}}.\]
This has rank 4 over $\W$ and generators are given by the residue classes of $e_1$, $ie_1$, $je_1$, $ke_1$.

Since $M_1 \cong  IPK^1$, as in Lemma 3.1.2 of \cite{BeaudryRes} we obtain that
\[H_1(PK^1, M_1) \cong H_2(PK^1, \W) =0.\]
From this, we get a short exact sequence
\[ 0 \to H_0(PK^1, M_2) \to H_0(PK^1, \sD_1) \to H_0(PK^1, M_1) \to 0.\]
As above, $H_0(PK^1, M_1) \cong H_1(PK^1, \W)$ and the latter is torsion. As a $\W$-module, $H_0(PK^1, M_2)$ must then be a free $\W$ sub-module of rank $4$ in $H_0(PK^1,  \sD_1)$.

To identify generators, we use the isomorphism 
\[\xymatrix{
H_0(PK^1, M_1) \ar[d]^-{\cong}  \ar[r]^-{\cong}  & H_1(PK^1, \W) \ar[d]^-\cong \\
I_{\Z_2}PK^1/(I_{\Z_2}PK^1)^2 \otimes_{\Z_2}\W \ar[r] &  \W/4\{\ba\} \oplus \W/2\{\ba_i,\ba_j\}}\]
which
maps $e-\alpha^{-1}$ to $-\overline{\alpha}$. So
\[ (\zeta(e-\alpha) + \zeta^{\sigma} (e-\alpha^{-1})) \mapsto \zeta \ba - \zeta^{\sigma} \ba = (1+2\zeta)\ba.\]
Using the formulas of \cref{lem:H1KZ2}, we get
\begin{align*}
\partial_1(e_1) &\mapsto - (1+2\zeta)\ba \\
\partial_1(ie_1) &\mapsto - (1+2\zeta)(\ba+\ba_i)  \\
\partial_1(je_1) &\mapsto  -(1+2\zeta) (\ba+\ba_j)  \\
\partial_1(ke_1) &\mapsto -(1+2\zeta)(3\ba+\ba_i+\ba_j) .
\end{align*}
At this point, the proof of Lemma 3.1.2 of \cite{BeaudryRes} applies in a straightforward way to prove that the kernel of $H_0(PK^1, \partial_1)$ is generated by 
\[f=(3+i+j+k)e_1  \in H_0(PK^1, M_2)  \]
as a $\WSG{PG_{48}}$-module. By the form of Nakayama's Lemma stated in Lemma 4.3 of \cite{ghmr},  $M_2$ is a cyclic $\W[\![P\G_2^1]\!]$-submodule of $\sD_1$ and any element of $M_2$ which reduces to $f$ in $H_0(PK^1, M_2)$ is a generator. Note that $M_2$ surjects onto $H_0(PK^1, M_2)$ so such elements exist. 
It follows that any element $\thetazero $ of $\W[\![P\G_2^1]\!]$ such that $\thetazero  e_1 \in \ker(\partial_1)$ and
\begin{align*} \thetazero  e_1 \equiv (3+i+j+k)e_1 \mod (4,  IPK^1)\end{align*}
is a generator. 
 This proves the claim.
\end{proof}

\begin{lemma}\label{lem:N2partial}
Let $\sD_2 = \W{\uparrow}^{P\GG_2^1}_{PG_{12}}$ with canonical generator $e_2$. There exists $\thetaphi  \in \WG{P\G_2^1}$ such that
\begin{enumerate}[(1)]
\item $\tau \thetaphi  = \thetaphi  \tau $ for any $\tau \in PG_{12}$,
\item $\thetaphi  e_1$ is in the kernel of $\partial_1$, and
\item $\thetaphi e_1 \equiv (3+i+j+k)e_1 \mod (4, IPK^1)$.
\end{enumerate}
If  $\thetaphi  $ is any element which satisfies these conditions, then
 the map of $\WG{P\G_2^1}$-modules $\widetilde{\partial}_2 \colon \sD_2 \to \sD_1$ given by
\[ \widetilde{\partial}_2(e_2) = \thetaphi e_1\]
is well-defined and
surjects onto $M_2 =\ker(\partial_1)$.

For any choice of element $\thetazero $ as in \Cref{lem:N1}, the element
\begin{equation}\label{eq:thetaphi}
 \thetaphi =-\frac{1}{3}\sum_{g\in PG_{12}} \zeta^{g}g\thetazero g^{-1}  
\end{equation}
satisfies the above conditions.
\end{lemma}
As the notation suggestion, the definition of $\widetilde{\partial}_2$ is temporary and will be modified to $\partial_2$ later in this section.

\begin{proof}[Proof of \cref{lem:N2partial}.]
Given $\thetaphi $ which satisfies (1), (2) and (3), the map $\widetilde{\partial}_2$ as defined above will be well-defined by (1) and it will surject onto $M_2$ by (2), (3) and \Cref{lem:N1}. So, it remains to prove existence. We do this by fixing an element $\thetazero $ as in \Cref{lem:N1} and showing that $\thetaphi $  as in \eqref{eq:thetaphi} does the job.
For $\tau \in PG_{12}$, we have
\[\tau \thetaphi  =-\frac{1}{3}\sum_{g\in PG_{12}} \tau \zeta^g g\thetazero g^{-1}\tau^{-1} \tau=-\frac{1}{3}\sum_{g\in PG_{12}}  \zeta^{\tau g}\tau g\thetazero g^{-1}\tau^{-1} \tau  = \thetaphi  \tau.\]
So (1) holds. For (2),  using \cref{rem:taucom} and the fact that $g^{-1}e_0=e_0$ for $g\in PG_{12}$ to obtain the second equality, we have
\begin{align*}
\partial_1(\thetaphi  e_1) &=-\frac{1}{3} \sum_{g\in PG_{12}} \zeta^g g\thetazero g^{-1}  (\zeta(e-\alpha) + \zeta^{\sigma} (e-\alpha^{\sigma})) e_0 \\
&=-\frac{1}{3}\sum_{g\in PG_{12}}  \zeta^g g\thetazero   (\zeta(e-\alpha) + \zeta^{\sigma} (e-\alpha^{\sigma})) e_0 \\
&=-\frac{1}{3}\sum_{g\in PG_{12}}  \zeta^g g\partial_1(\thetazero ) =0.\end{align*}
For (3), we will use that in $P\G_{2}^1$, we have
\begin{align*}
[j-k] i [j-k]^{-1}&= i &  [j-k] j [j-k]^{-1}&= k & [j-k] k [j-k]^{-1}&= j  \\
\omega i  \omega^{-1}&= j & \omega k \omega^{-1}&= k & \omega k \omega^{-1}&= i .
\end{align*}
We get for any $\tau \in PG_{12}$ that
\[ \tau(3+i+j+k) = (3+i+j+k)\tau \]
 in $\W[\![P\G_2^1]\!]$.
Modulo $ (4,  IPK^1)$, we then compute
\begin{align*}
\thetaphi e_1 &=-\frac{1}{3} \sum_{g\in PG_{12}}  \zeta^g g\thetazero  e_1 \\
&\equiv \sum_{g\in PG_{12}}  \zeta^g g(3+i+j+k)e_1 \\
&\equiv  \zeta  \sum_{g\in PC_6} g(3+i+j+k)e_1 + \zeta^\sigma \sum_{g\in PC_6}  [j-k] g(3+i+j+k)e_1  \\
&\equiv 3 \left(\zeta  (3+i+j+k)e_1 + \zeta^\sigma (3+i+j+k)e_1\right)  \\
&\equiv  (3+i+j+k)e_1 . \qedhere
\end{align*}
\end{proof}

\begin{lemma}\label{lem:M3}
Let $M_3  = \ker(\widetilde{\partial}_2)$ for $\widetilde{\partial}_2$ as in \Cref{lem:N2partial}, so that there is an exact sequence of $\WW[\![P\GG_2^1]\!]$-modules
\begin{equation}\label{eq:M3}
\xymatrix{ 0 \ar[r] &  M_3  \ar[r] &  \sD_2 \ar[r]^-{\widetilde{\partial}_2} \ar[r] & M_2 \ar[r] & 0 .}
\end{equation}
There is an isomorphism $M_3 \cong \WG{PK^1}$ as $\WG{PK^1}$-modules.
\end{lemma}
\begin{proof}
Note
\[H_0(PK^1,  \sD_2) \cong \WSG{PG_{48}/PG_{12}} \cong H_0(PK^1,  M_2 ) \]
as a $\WSG{PG_{48}}$-modules and the map $\widetilde{\partial}_2$ is designed to induce such an isomorphism. Indeed, $\widetilde{\partial}_2$ maps the image of the canonical generator $e_2$ in $H_0(PK^1,  \sD_2)$ to a generator $f$ of $H_0(PK^1,  M_2 )$.
Furthermore, $\sD_0$, $\sD_1$ and $\sD_2$ are free as $\WG{PK^1}$-modules. Together, these facts imply that for $n\geq 0$, there are isomorphisms
\begin{align*}H_n(PK^1, M_3) \cong H_{n+3}(PK^1, \W) \cong \begin{cases} \W & n=0 \\
0 & n>0\end{cases}
\end{align*}
as $PG_{48}$-modules.
Choosing $e'\in M_3$ so that $e'$ is a $\W$-module generator for the homology $H_0(PK^1, M_3)$,
we construct a map
\[ \WG{PK^1} \to M_3\]
which sends $x$ to $xe'$. By Lemma 4.4 of \cite{HennCent}, this is an isomorphism.
\end{proof}

Splicing the exact sequence of Lemmas~\ref{lem:N0}, \ref{lem:N1} and \ref{lem:M3} together, we get an exact sequence
\begin{equation}\label{eq:M3exact}
\xymatrix{0 \ar[r] &  M_3 \ar[r] & \sD_2 \ar[r]^-{\widetilde{\partial}_2} & \sD_1 \ar[r]^-{\partial_1} &  \sD_0 \ar[r]^-{\varepsilon}  & \W \ar[r] & 0}. 
\end{equation}
It remains to identify $M_3$ with $\W {\uparrow}^{P\G_{2}^1}_{PG_{48}'} $, where $PG_{48}' = P\pi G_{28}\pi^{-1}$.  This is the trickiest part of the argument. We will need the following observation.
\begin{rem}
Note that as  $PK^1$-modules 
\[\W {\uparrow}^{P\G_{2}^1}_{PG_{48}'} \cong \WG{PK^1}\]
and so, by \Cref{lem:PD}, we have
\begin{align}\label{eq:cohsparse}
H_n(PK^1, \W {\uparrow}^{P\G_{2}^1}_{PG_{48}'}) \cong \begin{cases}\W & n=0 \\
0 & \text{otherwise}.
\end{cases}\end{align}
\end{rem}

\begin{lemma}\label{lem:iota}
Let $\iota \colon PK^1 \to P\G_2^1$ be the inclusion and 
\[\iota^* \colon   H^3(P\G_2^1, \W {\uparrow}^{P\G_{2}^1}_{PG_{48}'}) \to H^3(PK^1, \W {\uparrow}^{P\G_{2}^1}_{PG_{48}'}  )\]
be the induced map. Then
 $\iota^*$ factors via an isomorphism
\[\xymatrix{ H^3(P\G_2^1, \W {\uparrow}^{P\G_{2}^1}_{PG_{48}'})  \ar[r]^-{\iota_0^*}_-{\cong} & H^3(PK^1, \W {\uparrow}^{P\G_{2}^1}_{PG_{48}'} )^{PG_{48}}   . }  \]
In particular, there is a commutative diagram
\[\xymatrix{  H^3(P\G_2^1, \W {\uparrow}^{P\G_{2}^1}_{PG_{48}'}) \ar[d]^{\cong}  \ar[r]^-{\iota^*} & H^3(PK^1, \W {\uparrow}^{P\G_{2}^1}_{PG_{48}'} )  \ar[d]^{\cong}  \\
\Z_2 =\W^{\Gal} \ar[r]^-{\subset} & \W .} \]
\end{lemma}
\begin{proof}
Recall from \Cref{warn:KnotGal} and  \Cref{rem:symmetricgroup}. that 
\[P\G_2^1 \cong PK^1 \rtimes PG_{48} .\]
Consider the spectral sequence associated to this group extension:
\begin{equation}\label{eq:Lyndon-HSSS}
H^p(PG_{48}, H^q( PK^1, \W {\uparrow}^{P\G_{2}^1}_{PG_{48}'} )) \Longrightarrow H^{p+q}(P\G_2^1, \W {\uparrow}^{P\G_{2}^1}_{PG_{48}'}).
\end{equation}
By  \eqref{eq:cohsparse}, $H^q( PK^1, \W {\uparrow}^{P\G_{2}^1}_{PG_{48}'} )$ is zero unless $q=3$. Therefore,
\begin{align*}
H^{3}(P\G_2^1, \W {\uparrow}^{P\G_{2}^1}_{PG_{48}'}) & \cong H^0( PG_{48} , H^3( PK^1, \W {\uparrow}^{P\G_{2}^1}_{PG_{48}'} )) .
\end{align*}
The map 
$\iota_0^*$ is the edge homomorphism of the spectral sequence \eqref{eq:Lyndon-HSSS}. This proves the first claim.

Now we analyze the action  of ${PG_{48}}$ on $H^3(PK^1, \W {\uparrow}^{P\G_{2}^1}_{PG_{48}'}) \cong \W$. There is a commutative diagram
\[ \xymatrix{1  \ar[r] &  PG_{24} \ar[d]  \ar[r] &  PG_{48}  \ar[d] \ar[r] & \Z/2   \ar[d]^-{\cong} \ar[r] & 1\\
1 \ar[r] & P\mathbb{S}_2(\Gamma) \ar[r] & P \mathbb{G}_2(\Gamma) \ar[r] & \Gal(\F_4/\F_2) \ar[r] & 1}\]
Note that $A_4 \cong PG_{24} \subseteq  P\mathbb{S}_2(\Gamma)$ for $A_4$ the alternating group on 4 letters. We get an isomorphism of $PG_{24}$-modules
\[H^3(PK^1,  \W {\uparrow}^{P\G_{2}^1}_{PG_{48}'}) \cong H^3(PK^1, {\Z_2}\!\!\uparrow_{PG_{48}'}^{P\GG_2^1}) \otimes_{\Z_2}\W\] with $PG_{24}$ acting trivially on $\W$.
The action of $PG_{24}$ on $H^3(PK^1, {\Z_2}\! \! \uparrow_{PG_{48}'}^{P\GG_2^1}) \cong \Z_2$ is trivial since there are no non-trivial $\Z_2[PG_{24}]$-modules whose underlying $\Z_2$-module is free of rank one.  It follows that
\begin{align*}
H^3(PK^1, \W {\uparrow}^{P\G_{2}^1}_{PG_{48}'})^{PG_{48}}  & \cong \W^{PG_{48}/PG_{24}} .
\end{align*}
The action of $PG_{48}/PG_{24}$ on $H^3(PK^1, \W {\uparrow}^{P\G_{2}^1}_{PG_{48}'})$ is induced by the action of $P\mathbb{G}_2^1$, and therefore, $PG_{48}/PG_{24}$ 
acts on $\W$ like the Galois group $\Gal$.
\end{proof}

\begin{lemma}
There are surjections
\begin{align*}
\eta &\colon \Hom_{\WG{P\G_2^1}}(M_3, \W {\uparrow}^{P\G_{2}^1}_{PG_{48}'}) \to H^3(P\G_2^1,\W {\uparrow}^{P\G_{2}^1}_{PG_{48}'} )  \\
\eta' &\colon\Hom_{\WG{PK^1}}(M_3, \W {\uparrow}^{P\G_{2}^1}_{PG_{48}'}) \to H^3(PK^1,\W {\uparrow}^{P\G_{2}^1}_{PG_{48}'} ) 
\end{align*}
such that the following diagram commutes
\[\xymatrix{ \Hom_{\WG{P\G_2^1}}(M_3, \W {\uparrow}^{P\G_{2}^1}_{PG_{48}'}) \ar[d]^-{\iota^*}\ar[r]^-{\eta} &H^3(P\G_2^1,\W {\uparrow}^{P\G_{2}^1}_{PG_{48}'} )  \ar[d]^-{\iota^{*}} \\  
\Hom_{\WG{PK^1}}(M_3, \W {\uparrow}^{P\G_{2}^1}_{PG_{48}'}) \ar[r]^-{\eta'} & H^3(PK^1,\W {\uparrow}^{P\G_{2}^1}_{PG_{48}'} ) ,
 }\] 
where $\iota^*$ are the maps induced by the inclusion $PK^1 \to P\G_2^1$.
\end{lemma}
\begin{proof}
The maps $\eta$ and $\eta'$ are constructed as in the proof of \cite[Lemma 3.1.5]{BeaudryRes}.
We let $\mathcal{P}_p=\sD_p$ for $p=0,1,2$ and $\mathcal{P}_3=M_3$.
Consider the spectral sequences
\begin{align*}
\xymatrix{
E_1^{p,q} \cong \Ext_{\WG{P\G_2^1}}^q(\mathcal{P}_p, \W {\uparrow}^{P\G_{2}^1}_{PG_{48}'}) \ar[d]_-{\iota_*} \ar@{=>}[rr]  & & H^{p+q}(P\G_2^1, \W {\uparrow}^{P\G_{2}^1}_{PG_{48}'}) \ar[d]_-{\iota_*}  \\
F_1^{p,q} \cong \Ext_{\WG{PK^1}}^q(\mathcal{P}_p, \W {\uparrow}^{P\G_{2}^1}_{PG_{48}'})  \ar@{=>}[rr]  & &  H^{p+q}(PK^1, \W {\uparrow}^{P\G_{2}^1}_{PG_{48}'})}
\end{align*}
obtained from resolving the long exact sequence \eqref{eq:M3exact} into a double complex of projective $\WG{PK^1}$-modules. 
The differentials have degree $d_r \colon E_1^{p,q} \to E_1^{p+r, q-r+1}$ and $d_r \colon F_1^{p,q} \to F_1^{p+r, q-r+1}$. Then $\eta$ is the edge homomorphism
\[\eta \colon E_1^{3,0} \cong  \Hom_{\WG{P\G_2^1}}(M_3, \W {\uparrow}^{P\G_{2}^1}_{PG_{48}'})  \to   H^3(P\G_2^1,\W {\uparrow}^{P\G_{2}^1}_{PG_{48}'} ) \]
and $\eta'$ the edge homomorphism
\[\eta' \colon F_1^{3,0} \cong  \Hom_{\WG{PK^1}}(M_3, \W {\uparrow}^{P\G_{2}^1}_{PG_{48}'})  \to   H^3(PK^1,\W {\uparrow}^{P\G_{2}^1}_{PG_{48}'} ) \]
Since $\sD_p$ are projective $\WG{PK^1}$-modules, $F^{*,q}=0$ for $q>0$ so the edge homomorphism $\eta'$ is surjective.

To prove that $\eta$ is surjective, one sets up the argument exactly as in \cite[Lemma 3.1.5]{BeaudryRes}, replacing $\Z_2[\![P\mathbb{S}_2^1]\!]$ by $\WG{P\G_2^1}$. 
It suffices to prove that $E_1^{p,3-p}=0$ for $p=0,1,2$. 

Following the logic of that argument, we note that for $p=1,2$
\begin{align*}
E_1^{p,3-p} &\cong \Ext_{\WG{P\G_2^1}}^{3-p}(\W {\uparrow}^{P\G_{2}^1}_{PG_{12}},\W {\uparrow}^{P\G_{2}^1}_{PG_{48}'} ) \\
&\cong H^{3-p}(PG_{12}, \W {\uparrow}^{P\G_{2}^1}_{PG_{48}'}  ) .
\end{align*}
Since $PC_6 \cong C_3$ has order prime to $2$, we have an isomorphism
\[H^{3-p}(PG_{12}, \W {\uparrow}^{P\G_{2}^1}_{PG_{48}'}  )  \cong H^{3-p}(\Gal, (\W {\uparrow}^{P\G_{2}^1}_{PG_{48}'} )^{PC_6}).\]
Since $ (\W {\uparrow}^{P\G_{2}^1}_{PG_{48}'} )^{PC_6} \cong (\Z_2{\uparrow}^{P\G_{2}^1}_{PG_{48}'})^{PC_6} \otimes \W$ is a projective Galois module and $3-p>0$, these groups vanish and we conclude that
$E_1^{p,3-p} =0$ when $p=1,2$.

So, to show that the edge homomorphism is surjective, it's enough to prove that
\[E_1^{0,3} = \Ext_{\WG{P\G_2^1}}^3(\W {\uparrow}^{P\G_{2}^1}_{PG_{48}},\W {\uparrow}^{P\G_{2}^1}_{PG_{48}'} )=0.\]
First, we note 
\begin{align*}
E_1^{0,3} & \cong H^3(PG_{48}, \W {\uparrow}^{P\G_{2}^1}_{PG_{48}'} ).
\end{align*}
We use the extension $PG_{24} \to PG_{48} \to \Gal$. We have a spectral sequence
\[H^s( \Gal , H^r(PG_{24} , \W {\uparrow}^{P\G_{2}^1}_{PG_{48}'})) \Rightarrow H^{s+r}(PG_{48}, \W {\uparrow}^{P\G_{2}^1}_{PG_{48}'}) \]
and isomorphisms
\[ H^r(PG_{24} , \W {\uparrow}^{P\G_{2}^1}_{PG_{48}'}) \cong H^r(PG_{24} , \Z_2 {\uparrow}^{P\G_{2}^1}_{PG_{48}'})\otimes_{\Z_2} \W. \]
But as a $\Gal$-module,  $\W \cong \Z_2[\Gal]$, and using the fact that the induced and co-induced modules are isomorphic for finite groups, we get isomorphisms
\[H^s( \Gal , H^r(PG_{24} , \W {\uparrow}^{P\G_{2}^1}_{PG_{48}'}))  \cong \begin{cases}
H^r(PG_{24} , \Z_2 {\uparrow}^{P\G_{2}^1}_{PG_{48}'}) & s=0\\
0 & s>0.
\end{cases}\]
We thus get an isomorphism
\begin{align}\label{eq:S3iso}
H^{3}(PG_{48}, \W {\uparrow}^{P\G_{2}^1}_{PG_{48}'}) \cong  H^{3}(PG_{24} , \Z_2 {\uparrow}^{P\G_{2}^1}_{PG_{48}'}) .
\end{align}
However, as left $PG_{24}$-modules, $\Z_2 {\uparrow}^{P\G_2^1}_{PG_{48}'}  \cong \Z_2 {\uparrow}^{P\mathbb{S}_2^1}_{PG_{24}'}$ so
\[H^3(PG_{24}, \Z_2 {\uparrow}^{P\G_2^1}_{PG_{48}'}) \cong  H^3(PG_{24}, \Z_2 {\uparrow}^{P\mathbb{S}_2^1}_{PG_{24}'}) .\]
In the last part of the proof of Lemma 3.1.5 in \cite{BeaudryRes}, this last module is shown to be trivial.
\end{proof}

\begin{theorem}
There is an isomorphism of $\WG{P\G_2^1}$-modules
\[ \varphi  \colon M_3 \to \W {\uparrow}^{P\G_{2}^1}_{PG_{48}'} .\]
\end{theorem}
\begin{proof}
The proof is similar to that of Theorem 3.1.6 of \cite{BeaudryRes}. Remark that $PK^1$ is an oriented Poincar\'e duality group of dimension $3$. It follows that there is a natural isomorphism
\[ \xymatrix{H^{3-*}(PK^1, M) \ar[r]^-{\cap [PK^1] } & H_*(PK^1, M)}\]
for a choice of generator $[PK^1] \in H_3(PK^1, \Z_2) \cong \Z_2$ and a complete $\Z_2[\![PK^1]\!]$-module $M$. We let $[PK^1]$ also denote the image of this class in $H_3(PK^1,\W)$. Note further that for $M \cong \W \otimes_{\Z_2} M_1$ where $M_1$ is a $\ZG{PK^1}$-module with $PK^1$ acting diagonally (and therefore trivially on $\W$), $\cap [PK^1]$ is $\W$--linear. 

We let
\[\nu \colon H_3(PK^1, \W ) \to \W \otimes_{\WG{PK^1}} M_3\]
be the edge homomorphism of the spectral sequence 
\[\Tor^{\WG{PK^1}}_q( \W, \sD_p) \Longrightarrow H_{p+q}(PK^1, \W)\]
obtained from the long exact sequence \eqref{eq:M3exact}
and let 
\[ \Hom_{\W}(\W \otimes_{\WG{PK^1}}M_3,   \W \otimes_{\WG{PK^1}} \W{\uparrow}^{P\mathbb{G}_2^1 }_{PG_{48}'} ) \xrightarrow{\ev} H_0(PK^1, \W{\uparrow}^{P\mathbb{G}_2^1 }_{PG_{48}'} ) \]
be given by
\[\ev(f) =f(\nu([PK^1])). \]
By definition of the cap product, we have a commutative diagram
\[\xymatrix{ \Hom_{\WG{P\mathbb{G}_2^1}}(M_3,   \W{\uparrow}^{P\mathbb{G}_2^1 }_{PG_{48}'}  )  \ar@{->>}[r]^-{\eta} \ar[d]_-{\iota^*} & H^3(P\mathbb{G}_2^1,  \W{\uparrow}^{P\mathbb{G}_2^1 }_{PG_{48}'} ) \ar[d]^-{\iota^*}   \\
 \Hom_{\WG{PK^1}}(M_3,   \W{\uparrow}^{P\mathbb{G}_2^1 }_{PG_{48}'})  \ar@{->>}[r]^-{\eta'} \ar[d]_-{\W \otimes_{\WG{PK^1}} -} & H^3(PK^1,  \W{\uparrow}^{P\mathbb{G}_2^1 }_{PG_{48}'}) \ar[d]^-{\cap [PK^1]}_-{\cong}   \\
 \Hom_{\W}(\W \otimes_{\WG{PK^1}}M_3,   \W \otimes_{\WG{PK^1}} \W{\uparrow}^{P\mathbb{G}_2^1 }_{PG_{48}'}) \ar[r]^-{\ev}_-{\cong}   & H_0(PK^1, \W{\uparrow}^{P\mathbb{G}_2^1 }_{PG_{48}'}  ).   } \]
The map $\ev$ is an isomorphism. Indeed, it is a surjective morphism of $\W$-modules since $\eta'$ and $\cap [PK^1]$ are surjective and both sides are abstractly isomorphic to $\W$. 
 
 For a map $\varphi$, if $\W\otimes_{\WG{PK^1}} \varphi $ is an isomorphism, then Lemma 4.4 of \cite{HennCent} implies that $\varphi \in \Hom_{\WG{PK^1}}(M_3,   \W{\uparrow}^{P\mathbb{G}_2^1 }_{PG_{48}'} ) $ is an isomorphism. Choose a unit 
 \[\varphi_0 \in H_0(PK^1, \W{\uparrow}^{P\mathbb{G}_2^1 }_{PG_{48}'}  )^{\Gal} \cong \Z_2.\] 
 Since $(\cap [PK^1])\circ \iota^*\circ \eta $ surjects onto the Galois invariants, there is a 
 \[\varphi \in \Hom_{\WG{P\mathbb{G}_2^1}}(M_3,   \W{\uparrow}^{P\mathbb{G}_2^1 }_{PG_{48}'}  )\] such that $(\cap [PK^1]) \iota^* \eta (\varphi) = \varphi_0$. Since $\ev^{-1}(\varphi_0)= \W \otimes_{\WG{PK^1}} \iota^*(\varphi)$ is an isomorphism, $\iota^*(\varphi)$ is an isomorphism, and hence so was $\varphi$.
\end{proof}

\begin{defn}
Let $\sD_3 =  \W{\uparrow}^{P\mathbb{G}_2^1 }_{PG_{48}'}$.
\[ \widetilde{\partial}_3 \colon \sD_3  \to  \sD_2 \]
be the composition of $\varphi^{-1} \colon   \sD_3 \to M_3 $ with the inclusion of $M_3$ in $ \sD_2$.
\end{defn}

As the notation suggests, the map $\widetilde{\partial}_3$ is temporary and our next goal is to give a better description of the last map.
To do this, we address the duality of the resolution. Just as in \cite[\S3.3]{BeaudryRes} we have an involution
\[c_{\pi} \colon \Mod^\phi(P\GG_2^1) \to  \Mod^\phi(P\GG_2^1) \]
which twists the $P\GG_2^1$-action on a module by conjugation by $\pi$. That is, if $M$ is a Galois-twisted $\WG{P\GG_2^1}$-module, then $c_\pi(M)$ is the module whose $\W$-module is equal to $M$, but where the action is given by
\[a \cdot m = \pi a\pi^{-1}m.\]
If $\varphi \colon M \to N$ is a morphism, let 
\[c_\pi(\varphi) (m) = \varphi(  m).\]
For a Galois-twisted module $M$, we let 
\[M^* = \Hom_{\W[\![P\GG_2^1]\!]}(M, \W[\![P\GG_2^1]\!]).\]
This is a $\W[\![P\GG_2^1]\!]$-module where the action on a function $f$ is given by  $(a f)(x) = f(x)ae$ for $a\in \W$ and 
$(gf)(x) = f(x)g^{-1}$ for $g\in P\GG_2^1$. 
If $M = \W\!\uparrow_{H}^{P\GG_2^1}$ for $H$ a finite subgroup of $P\GG_2^1$, then there is an isomorphism
\[ \W\!\uparrow_{H}^{P\GG_2^1} \to \left(\W\!\uparrow_{H}^{P\GG_2^1}\right)^* \]
which is the map of  $\W[\![P\GG_2^1]\!]$-modules which sends the coset $[g]$ to
\begin{align}\label{eq:dualformula}
[g]^*(x) = x \sum_{h\in H} hg^{-1}.
\end{align}

To compute the last map, one uses the duality of the resolution, which we state here. The proof is completely analogous to that of \cite[Theorem 3.3.1]{BeaudryRes}.
\begin{theorem}
There exists isomorphisms of complexes of left $\W[\![P\GG_2^1]\!]$-modules
\[\xymatrix@C=1.5pc{0 \ar[r] & \sD_3 \ar[r]^-{\widetilde{\partial}_3} \ar[d]^-{f_3} & \sD_2 \ar[r]^-{\widetilde{\partial}_2} \ar[d]^-{f_2}  &  \sD_1 \ar[r]^-{\partial_1} \ar[d]^-{f_1}  &  \sD_0 \ar[r] \ar[d]^-{f_0} & \W \ar[r] \ar[d]^-{=} & 0 \\
0\ar[r] &  c_\pi((\sD_0)^*) \ar[r]^-{c_\pi(\partial_1^*)}&  c_\pi((\sD_1)^*)  \ar[r]^-{c_\pi(\widetilde{\partial}_2^*)}&  c_\pi((\sD_2)^*)  \ar[r]^-{c_\pi(\widetilde{\partial}_3^*)}&  c_\pi((\sD_3)^*) \ar[r] & \W \ar[r] &0  .
 }\]
\end{theorem}

\begin{lem}\label{lem:delta3prime}
Let $e_3$ be the canonical generator of $ \sD_3 $. There are automorphisms $g_2 \colon \sD_2 \to \sD_2$ and $g_3 \colon \sD_3 \to \sD_3$ such that
\[\widetilde{\partial_3}= g_2^{-1} \partial_3' g_3 \]
for  $\partial_3' \colon \sD_3 \to \sD_2$ defined by
\[\partial_3'(e_3)  = \tr_\sigma(\delta_3)e_2=  \tr_{\sigma}( \pi(e+i+j+k)(e-\alpha^{-1})\pi^{-1}) e_2 . \]
Here, $\delta_3 =  \pi(e+i+j+k)(e-\alpha^{-1})\pi^{-1}$.
\end{lem}

\begin{proof}
The argument mirrors the proof of Theorem 3.4.1 in \cite{BeaudryRes}. 
We have isomorphisms
\[ q_p \colon \sD_{p}  \xrightarrow{\cong} (\sD_{p})^* \]
which sends $e_p$ to $e_p^*$ as in \eqref{eq:dualformula}. We call the element corresponding to  $e_p^*$ in $c_\pi((\sD_p)^*) $ by $e_p^\pi$.
We will construct a commutative diagram
\[\xymatrix@C=1.5pc{0 \ar[r] & \sD_3 \ar[r]^-{\widetilde{\partial}_3} \ar[d]^-{f_3} & \sD_2 \ar[r]^-{\widetilde{\partial}_2} \ar[d]^-{f_2}  &  \sD_1 \ar[r]^-{\partial_1} \ar[d]^-{f_1}  &  \sD_0 \ar[r] \ar[d]^-{f_0} & \W \ar[r] \ar[d]^-{=} & 0 \\
0\ar[r] &  c_\pi((\sD_0)^*) \ar[r]^-{c_\pi(\partial_1^*)} \ar[d]^-{q_3} &  c_\pi((\sD_1)^*)  \ar[r]^-{c_\pi(\widetilde{\partial}_2^*)} \ar[d]^-{q_2} &  c_\pi((\sD_2)^*)  \ar[r]^-{c_\pi(\widetilde\partial_3^*)} \ar[d]^-{q_1} &  c_\pi((\sD_3)^*) \ar[r] \ar[d]^-{q_0} & \W \ar[r] \ar[d]^-{=} &0  \\
0\ar[r] &  \sD_3  \ar[r]^-{\partial_3'} &  \sD_2 \ar[r]^-{\partial_2'} &  \sD_1 \ar[r]^-{\partial_1'}&  \sD_0   \ar[r] & \W \ar[r] &0  
 }\]
with the vertical maps
\[q_p \colon  c_\pi((\sD_{3-p})^*)  \to \sD_p \]
defined  by
\[q_p(e_{3-p}^\pi)=e_p\]
and $\partial_{p+1}'\colon   \sD_{p+1}  \to  \sD_p $ by 
\[\partial_{p+1}' = q_pc_\pi(\partial_{3-p}^*)q_{p+1}^{-1}.\]
By construction, our diagram is commutative. To compute $\partial_3'$, we need to compute $\partial_1^*$, which we do using \eqref{eq:dualformula}:
\begin{align*}
\partial_1^*(e_0^*)(e_1) &= e_0^*\circ \partial_1(e_0)  \\
 &= e_0^*( (-\zeta (e-\alpha)- \zeta^{\sigma} (e-\alpha^{\sigma})  )e_0)  \\
  &=  -\zeta (e-\alpha) e_0^*(e_0)- \zeta^{\sigma} (e-\alpha^{\sigma})  e_0^*(e_0)  \\
&=(-\zeta (e-\alpha)- \zeta^{\sigma} (e-\alpha^{\sigma}) ) \sum_{h\in PG_{48}} h \\
&= \sum_{h\in PG_{12} }  h (-\zeta (e-\alpha)- \zeta^{\sigma} (e-\alpha^{\sigma}) )(e+i^{-1}+j^{-1}+k^{-1}) \\
&=((e+i+j+k)(-\zeta (e-\alpha^{-1})- \zeta^{\sigma} (e-(\alpha^{-1})^{\sigma} )e_1^*)(e_1).
\end{align*}
Therefore,
\begin{align*}
 \partial_1^*(e_0^*)&=(e+i+j+k)(-\zeta (e-\alpha^{-1})- \zeta^{\sigma} (e-(\alpha^{-1})^{\sigma} )e_1^*  \\
 &=(e+i+j+k)\tr_{\sigma}(e-\alpha^{-1})e_1^* 
 \end{align*}
From this and a diagram chase, we get the formula
\[\partial_3'(e_3)  = \pi(e+i+j+k)  \tr_{\sigma}(e-\alpha^{-1})\pi^{-1} e_2 . \]
The identities $\pi\pi^{\sigma}=3$ and $\pi^2=-3$ imply that $\pi, \pi^\sigma, \pi^{-1}$ and $(\pi^\sigma)^{-1}$ differ by central elements, and so their conjugation action agrees.
Since $(e+i+j+k)^\sigma = (e+i+j+k)$ in $\WG{P\G_2^1}$, we deduce that
\[  \tr_\sigma(\delta_3) = \pi(e+i+j+k)\tr_\sigma(e-\alpha^{-1}) \pi^{-1}.\]
This establishes the claim for $\partial_3'$.
Now, letting $g_p= q_pf_p$, we have
\[\widetilde{\partial}_3 = g_2^{-1} \partial_3' g_3. \qedhere\]
\end{proof}

\begin{defn}
Let $\partial_3 \colon \sD_3 \to \sD_2$ be defined as $\partial_3=\partial_3'$ and $\partial_2\colon \sD_2 \to \sD_1$ as $\partial_2 = \widetilde{\partial}_2g_2^{-1}$.
\end{defn}
\noindent
This completes the construction of the resolution.

In the next section, we will want to compare this resolution with that of \cite{BeaudryRes}. To do this, we note
 that in \cite{BeaudryRes}, the author could have proved the analogous result as above by redefining, as in the notation of that paper, the last map as $\partial_3'$ and the middle map as $\partial_2g_2^{-1}$ for $g_2$ now as in \cite[Theorem 3.4.1]{BeaudryRes}. We wish to state this here, but to do so, first recall:
 \begin{lem}[{\cite[Lem. 3.1.3]{BeaudryRes}}]\label{thm:thetaall}
There is an $\thetaclass \in\Z_2[\![P\mathbb{S}_2^1]\!]$ which satisfies the following three conditions:
\begin{enumerate}[(1)]
\item $\tau \thetaclass =\thetaclass\tau$ for all $\tau \in PC_6$,
\item $\thetaclass \delta_1 e_0=0$,
\item $\thetaclass \equiv 3+i+j+k \mod (4, IPK^1)$.
\end{enumerate}
\end{lem}

Combining the results shown so far in this section and their analogues for the resolution of \cite{BeaudryRes}, we have now shown:
\begin{cor}\label{cor:resbetterclassical}
We have:
\begin{enumerate}[(a)]
\item
There is an exact sequence of $\Z_2[\![P\SS_2^1]\!]$-modules
\[0 \rightarrow   \Z_2{\uparrow}^{P\Sn_{2}^1}_{PG_{24}'}  \xra{{\partial}_3}  \Z_2 {\uparrow}^{P\Sn_{2}^1}_{PC_{6}}  \xra{{\partial}_2} \Z_2 {\uparrow}^{P\Sn_{2}^1}_{PC_{6}}  \xra{\partial_1}  \Z_2 {\uparrow}^{P\Sn_{2}^1}_{PG_{24}} \xra{\varepsilon} \Z_2 \rightarrow 0,
\]
where
\begin{enumerate}[(1)]
\item $\partial_1(e_1)   = \delta_1e_0 $ for $\delta_1 = (e-\alpha)$,
\item   $\partial_2  = \widetilde{\partial}_2 g_2^{-1}$ where  $g_2$ is an automorphism of $ \Z_2 {\uparrow}^{P\Sn_{2}^1}_{PC_{6}} $ and $ \widetilde{\partial}_2(e_2) = \thetaclass e_1$ for $\thetaclass$ as in \cref{thm:thetaall}, 
\item   $\partial_3(e_4)   = \delta_3e_2 $ for $\delta_3 =  \pi (e+i+j+k)(e-\alpha^{-1}) \pi^{-1}$.
\end{enumerate}
\item There is an exact sequence of $\WG{P\GG_2^1}$-modules
\[0 \to  \W {\uparrow}^{P\G_{2}^1}_{PG_{48}'}\xra{\partial_3}  \W {\uparrow}^{P\G_{2}^1}_{PG_{12}}  \xra{\partial_2}  \W {\uparrow}^{P\G_{2}^1}_{PG_{12}}  \xra{\partial_1}  \W {\uparrow}^{P\G_{2}^1}_{PG_{48}} \xra{\varepsilon} \W \to 0 \]
where 
\begin{enumerate}[(1)]
\item  $\partial_1(e_1) =\delta_1^\phi e_0$ for $\delta_1^\phi =  \tr_\sigma(\delta_1)$,
\item   $\partial_2  = \widetilde{\partial}_2 g_2^{-1}$ where  $g_2$ is an automorphism of $ \W {\uparrow}^{P\G_{2}^1}_{PG_{12}}  $ and $ \widetilde{\partial}_2(e_2) = \thetaphi e_1$ for $\thetaphi$ as in \cref{lem:N2partial}, 
\item $\partial_3(e_3) =\delta_3^\phi e_2$ for $\delta^\phi_3  =  \tr_\sigma(\delta_3)$
\end{enumerate}
for  $\delta_1, \delta_3$ as (a).
\end{enumerate}
\end{cor}

\subsection{A description of the middle map}\label{sec:middle}
Our next goal is to improve on the description of $\partial_2$ in the above result \cref{cor:resbetterclassical} and to relate the map in (b) to that in (a).
In both cases, the middle map $\partial_2$ is the composite of two maps, $\widetilde{\partial}_2$ and $g_2^{-1}$ and we will study them separately.

We start with examining the automorphisms $g_2$.
For this, we note the following general result about the endomorphisms of $  \Z_2\!\uparrow_{PC_6}^{P\SS_2^1}$ as a  $ \Z_2[\![P\SS_2^1]\!]$-module and endomorphisms of  $ \W{\uparrow}^{P\GG_2^1}_{PG_{12}} $ as a Galois-twisted $\WG{P\mathbb{G}_2^1}$-module.  
Recall that
\begin{align*} 
H_0(PK^1,   \Z_2\!\uparrow_{PC_6}^{P\SS_2^1}) \cong \Z_2\!\uparrow_{PC_{6}}^{PG_{24}} \\
H_0(PK^1,  \W{\uparrow}^{P\GG_2^1}_{PG_{12}}) \cong \W\!\uparrow_{PG_{12}}^{PG_{48}}
\end{align*}
and that these are free rank four as $\Z_2$-, respectively $\W$-modules, generated by the images of $e,i, j, k$. 
\begin{lem}\label{lem:upsilon}
We have:
\begin{enumerate}[(a)]
\item For an endomorphism of $ \Z_2[\![P\SS_2^1]\!]$-modules
\[\varphi \colon   \Z_2\!\uparrow_{PC_6}^{P\SS_2^1}  \to  \Z_2\!\uparrow_{PC_6}^{P\SS_2^1} ,\]
there is an element $\upsilon \in \Z_2[\![P\SS_2^1]\!]$  such that $\tau\upsilon =\upsilon\tau $ for all $\tau \in   PC_6$,
\begin{align*}
\varphi(e) &=\upsilon e \\
\upsilon e &\equiv \lambda  e + \mu (i+j+k)  \mod IPK^1
\end{align*}
for $\lambda , \mu \in \Z_2$. 
\item  For and endomorphism of Galois twisted $\WG{P\GG_2^1}$-modules
\[\varphi \colon   \W\!\uparrow_{PG_{12}}^{P\GG_2^1}  \to  \W\!\uparrow_{PG_{12}}^{P\GG_2^1} , \]
there is an element $\upsilon_\phi \in \WG{P\GG_2^1}$  such that $\tau\upsilon_\phi =\upsilon_\phi\tau  $ for all $\tau \in   PG_{12}$,
\begin{align*}
\varphi(e) &=\upsilon_\phi e \\
\upsilon_\phi e &\equiv \lambda  e + \mu (i+j+k)  \mod IPK^1
\end{align*}
for some $\lambda , \mu \in \Z_2$. 
\end{enumerate}
In both cases, $\varphi$ is an automorphism if and only if 
\[\lambda  -\mu \not\equiv 0\mod 2.\]
\end{lem}

\begin{proof}
We prove (b), the proof for (a) being completely analogous and easier.
Choose any $\upsilon_0 \in \W[\![P\G_2^1]\!]$ such that
\[\varphi(e) = \upsilon_0 e.\]
For any $\tau \in PG_{12}$, we must have
\[ \tau\upsilon_0 e =\upsilon_0e . \]
Let
\[\upsilon_\phi =-\frac{1}{3}\sum_{g\in PG_{12}} \zeta^gg\upsilon_0 g^{-1}  .\]
Then
\begin{align*}
\upsilon_\phi  e &=-\frac{1}{3}\sum_{g\in PG_{12}} \zeta^gg\upsilon_0 g^{-1}e  \\
&=-\frac{1}{3}\sum_{g\in PG_{12}} \zeta^g\upsilon_0 e  \\
&=\upsilon_0 e  .
\end{align*}
So, $\varphi(e) = \upsilon_\phi  e$. 
For some  $\lambda, \lambda_i, \lambda_j,  \lambda_k \in \W$, we have
\[\upsilon_\phi  e \equiv \lambda e+ \lambda_i i + \lambda_j j  +\lambda_k   k \mod I  PK^1.\]
Using the invariance with respect to $\omega$ we get $\lambda_i =\lambda_j  = \lambda_k$, and we let this common value be $\mu$. The invariance with respect to $[j-k]$ we get $\lambda,\mu \in \Z_2 $. 

By Lemma 4.3 of \cite{ghmr}, $\upsilon$ is an isomorphism if and only if $\Tor_{0}^{PS_2^1}(\F_4, \varphi)$ is an isomorphism and $\Tor_{1}^{PS_2^1}(\F_4, \varphi)$ is onto. 
The first condition gives the requirement that $\lambda-\mu \not\equiv 0 \mod 2$, and the second is automatically satisfied since
$  \W\!\uparrow_{PG_{12}}^{P\GG_2^1} $ is projective as a $PS_2^1$-module. 
\end{proof}

Recall that  the ideal $\Iideal$ was defined in \eqref{defn:idealI}. For $P\G_2^1$, it corresponds to
\[\Iideal=(4, 2(IPS_2^1)^2, 2IPK^1, (IPK^1)^2, IPK^1 \cdot (IPS_2^1)^2, (IPS_2^1)^2 \cdot IPK^1 ) .\]
\begin{lem}\label{lem:thetaupsilon}
We have:
\begin{enumerate}[(a)]
\item There is an element $\delta_2 \in \Z_2[\![P\SS_2^1]\!]$ such that $\tau \delta_2  = \delta_2 \tau$ for all $\tau \in PC_{6}$, such that
\[ \partial_2 (e_2) =\delta_2 e_1\]
and, for $\thetaclass$ as in \cref{thm:thetaall},
\[ \delta_2 e_1 \equiv \thetaclass e_1 \mod \Iideal .\]
\item There is an  element $\delta_2^\phi \in \WG{P\GG_2^1}$ such that $\tau \delta_2^\phi  = \delta_2^\phi \tau$ for all $\tau \in PG_{12}$, such that
\[ \partial_2 (e_2) =\delta_2^\phi e_1\]
and, for $\thetaphi$ as in \cref{lem:N2partial},
\[ \delta_2^\phi e_1 \equiv \thetaphi  e_1 \mod \Iideal .\]
\end{enumerate}
\end{lem}
\begin{proof}
Again, we prove (b) and (a) follows similarly.
 Let $\upsilon_\phi$ be as in \cref{lem:upsilon} for $\varphi = g_2^{-1}$, so that $\varphi(e_2) = \upsilon_\phi e_2$.
Then
\begin{align*}
\partial_2(e_2)&=  \widetilde{\partial}_2(g_2^{-1}(e_2)) \\
&=\upsilon_\phi \thetaphi   e_1 .
\end{align*}
Let $\delta_2^\phi= \upsilon_\phi \thetaphi  $.  This proves the first claim.

Note that
\[\partial_2(e_2) =  \upsilon_\phi \thetaphi   e_1 = \thetaphi  e_1 + ( \upsilon_\phi - 1)   \thetaphi  e_1  \]
so we need to control $( \upsilon_\phi - 1)  \thetaphi  e_1$ for the second claim.
Recall that
\[ \thetaphi  e_1 \equiv  (3+i+j+k)e_1  \mod (4, I  PK^1). \]
 Write
 \[\upsilon_\phi  e_1 \equiv (\lambda e+\mu(i+j+k)) e_1 \mod IPK^1 \]
  as in \cref{lem:upsilon}.
Since $g_2^{-1}$ is an isomorphism,
\[(\lambda -1)  \equiv \mu \mod 2\]
so that 
\[(\lambda -1) e+\mu(i+j+k) \equiv \mu(e+i+j+k)  \mod 2.\]
Furthermore,
\begin{align*} 
e+i+j+k &\equiv (1-i)(1-j) \mod 2
\end{align*}
in $\WG{P\G_2^1}$. So $(e+i+j+k)\in  (2, (IS_2^1)^2)$. Therefore,
\[ (e+i+j+k)   (3+i+j+k) = 6(e+i+j+k)  \in (4,2(IS_2^1)^2).  \]
It now follows from these facts and  a direct computation that
\begin{align*}
( \upsilon_\phi - 1)   \thetaphi  e_1  &\equiv (\mu(e+i+j+k)  + (2,IPK^1)) (3+i+j+k + (4, I  PK^1))\\
&\equiv 0 \mod \Iideal. \qedhere
\end{align*}
\end{proof}

We turn to studying the relationship between the elements $\thetaclass$ of \cref{thm:thetaall} and  $\thetaphi$  of \cref{lem:N2partial}. 
Our goal is to prove that $\thetaphi $ can be approximated by $\tr_{\sigma}(\thetaclass)$.
\begin{prop}\label{lem:thetaphibest}
For $\thetaclass \in \Z_2[\![P\mathbb{S}_2^1]\!]$ as in \cref{thm:thetaall}, the element $\thetaphi $ of \cref{lem:N2partial} can be chosen to satisfy
\[ \thetaphi \equiv \tr_\sigma(\thetaclass) \mod (4IPK^1, 2(IPS_2^1)^2\cdot IPK^1, (IPK^1)^2) \subset \Iideal.  \]
\end{prop}
\begin{proof}
Recall that $\delta_1=e-\alpha$. Condition (2) of \cref{thm:thetaall} implies that
\[\thetaclass \delta_1e_0=0.\]
Therefore, $\thetaclass^\sigma \delta_1^\sigma e_0=0$ as well and
\begin{align}\label{eq:idtrtr}
\tr_\sigma(\thetaclass) \tr_\sigma(\delta_1)e_0  &= (\zeta \thetaclass +\zeta^{\sigma} \thetaclass^\sigma)(\zeta \delta_1 +\zeta^{\sigma} \delta_1^\sigma)e_0 \\
&= (\thetaclass\delta_1^\sigma+\thetaclass^\sigma \delta_1 )e_0  \nonumber \\
&= (\thetaclass +\thetaclass^\sigma )( \delta_1 + \delta_1^\sigma)e_0. \nonumber
\end{align}
We have in $\WG{P\GG_2^1}$, 
\[ \delta_1 + \delta_1^\sigma = \tr_\sigma(e-\alpha) +  \tr_\sigma(e-\alpha^{\sigma}) .\]
But,
\begin{align*}
e-\alpha^{\sigma} &=  (e-\alpha^{-1}) \\
&= e- \frac{1}{e-(e-\alpha)} \\
&= e- \sum_{n\geq 0} (e-\alpha)^n \\
&= - \sum_{n\geq 1} (e-\alpha)^n \\
&= -(e-\alpha)-  \sum_{n\geq 2} (e-\alpha)^n.
\end{align*}
So, 
\begin{align*} \delta_1 + \delta_1^\sigma & =   -\sum_{n\geq 2} \tr_\sigma ( (e-\alpha)^n) \\
& \equiv  0\mod  (IPK^1)^2.
 \end{align*}
It follows from \cref{lem:N0} that we can write 
\[(\delta_1 + \delta_1^\sigma)e_0 =\gamma \tr_{\sigma}(\delta_1)e_0 \]
for $\gamma \in \WG{P\mathbb{G}_2^1}$ which is in $IPK^1$, and thus, using \eqref{eq:idtrtr}, we get
\[ (\tr_\sigma(\thetaclass) - (\thetaclass +\thetaclass^\sigma)\gamma  )e_1 \in \ker(\partial_1).\]
So, we can take
\[\thetazero  =  \tr_\sigma(\thetaclass) - (\thetaclass +\thetaclass^\sigma)\gamma \]
in \cref{lem:N1}.
Since $(3+i+j+k)^\sigma =3+i+j+k$ in $\WG{P\G_2^1}$, condition (3) implies that 
\[\thetaclass + \thetaclass^{\sigma} \equiv 2\thetaclass \mod (4,IPK^1).\]
Furthermore,
\[\thetaclass \equiv (1-i)(1-j) \mod (2,IPK^1)\] 
and  so
\[\thetaclass + \thetaclass^{\sigma} \in (4, 2(IS_2^1)^2, IPK^1).\]
Since $\gamma\in IPK^1$, 
\[\thetazero  \equiv \tr_\sigma(\thetaclass)  \mod (4IPK^1, 2(IPS_2^1)^2\cdot IPK^1, (IPK^1)^2).\]
The associated $\thetaphi $ as in \eqref{eq:thetaphi} of \cref{lem:N2partial} is given by
\[\thetaphi  =- \frac{1}{3}\sum_{g\in PG_{12}} \zeta^g g \thetazero g^{-1} \]
but $\tr_\sigma(\thetaclass) $ is left unchanged by this averaging procedure by condition (2). So
\[\thetaphi  \equiv \tr_\sigma(\thetaclass) \mod (4IPK^1, 2(IPS_2^1)^2\cdot IPK^1, (IPK^1)^2). \qedhere\]
\end{proof}

Combining \cref{lem:thetaupsilon} and  \cref{lem:thetaphibest} directly shows:
\begin{theorem}\label{cor:delta2better}
For the ideal $\Iideal$ as in \eqref{defn:idealI}:
\begin{enumerate}[(a)]
\item The maps $\partial_p$ of \cref{cor:resbetterclassical}(a) satisfy
\[\partial_p(e_p) = \delta_pe_{p-1} \]
for $\delta_1,\delta_2,\delta_3 \in \Z_2[\![P\SS_2^1]\!]$ such that
\begin{align*}
\delta_1e_0  &=  (e-\alpha)e_0\\
\delta_2e_1 &\equiv \thetaclass e_1 \mod \Iideal\\
\delta_3e_2&= \pi(e+i+j+k)(e-\alpha^{-1})\pi^{-1}e_2.
\end{align*}
\item 
The maps $\partial_p$ satisfy of \cref{cor:resbetterclassical}(b) satisfy
\[\partial_p(e_p)  = \delta_p^\phi e_{p-1}\]
for  $\delta_1^\phi, \delta_2^\phi, \delta_3^\phi \in \WG{P\GG_2^1}$ such that
\begin{align*}
\delta_1^\phi e_0&= \tr_\sigma(\delta_1) e_0 \\
\delta_2^\phi e_2 &\equiv \tr_\sigma(\delta_2)e_1 \mod \Iideal \\
\delta_3^\phi e_3 &= \tr_\sigma(\delta_3) e_2 
\end{align*}
for $\delta_1,\delta_2, \delta_3$ as in (a).
\end{enumerate}
\end{theorem}

We end this section by  giving an explicit formula for $\delta_2$ modulo $\Iideal$. From this, we get an explicit formula for $\delta_2^\phi \equiv \tr_{\sigma}(\delta_2)$ modulo  $\Iideal$. First, we recall the best known approximation for $\thetaclass$.
\begin{theorem}[{\cite[Thm. 3.4.5]{BeaudryRes}}]\label{thm:thetasuperprecise}
There exists $\thetaclass$ in $\Z_2[\![P\mathbb{S}_2^1]\!]$ which satisfies the conditions of \cref{thm:thetaall} such that
\begin{align*}
\thetaclass  &\equiv e + \alpha+i+j+k -\alpha_i-\alpha_j -\alpha_k \\
& -\frac{1}{3} \tr_{C_3}\bigg((e-\alpha_i)(j-\alpha_j)+(e-\alpha_i\alpha_j) (k-\alpha_k)+(e-\alpha_i\alpha_j\alpha_k)(e+\alpha)\bigg)
\end{align*}
modulo $((IPK^1)^7, 2(IPK^1)^3, 4 (IPK^1), 8)$. Here, $ \tr_{C_3} (x) = x+\omega x \omega^{-1} +\omega^2x\omega^{-2}$. 
\end{theorem}
We use this to give our best approximation for $\delta_2$ and $\delta_2^\phi$, which also proves the simpler congruences of  \cref{thm:classicdualres} and in \cref{thm:updatedualres} modulo the ideal 
\[ \Jideal=(2, I  PK^1 \cdot I  PS_2^1,  I  PS_2^1 \cdot I  PK^1 )   .\]
Note that $\Iideal\subset\Jideal$.
\begin{cor}\label{cor:finalbestapproxdelta2}
Modulo $\Iideal$,
\begin{align*}
\delta_2 & \equiv e + \alpha+i+j+k -\alpha_i-\alpha_j -\alpha_k \\
& -\frac{1}{3} \tr_{C_3}\bigg((e-\alpha_i)(j-\alpha_j)+(e-\alpha_i\alpha_j) (k-\alpha_k)\bigg) .
\end{align*}
Modulo $\Jideal$, 
\[ \delta_2e_1 \equiv (\alpha +i+j+k)e_1.\]
\end{cor}
\begin{proof}
We have
\[((IPK^1)^7, 2(IPK^1)^3, 4 (IPK^1), 8)\subset \Iideal \subset \Jideal .\] 
Since $(2IPK^1, (IPK^1)^2)\subset \Iideal$,  
\[(e-\alpha_i\alpha_j\alpha_k)(e+\alpha)\equiv  (e-\alpha_i\alpha_j\alpha_k)(e-\alpha) \equiv 0 \mod \Iideal,\] 
so we get the first claim.
Working modulo $\Jideal$  and using the fact that
\begin{align*}
 (\tau-\alpha_t)(e-\alpha) = (e-\alpha_\tau \alpha)(\tau-e)+(e-\alpha_\tau)
\end{align*}
in $\Z_2[\![P\mathbb{S}_2^1]\!]$, the formula simplifies to
\[ \delta_2e_1 \equiv (\alpha +i+j+k)e_1 \mod \Jideal. \qedhere\]
\end{proof}

 % !TEX root = duality-master.tex

\section{Topological duality resolution}\label{sec:topres}
In this section, we use the resolution  of \cref{thm:updatedualres} to construct a resolution of spectra in the sense of \cite{HennRes}. We first recall the definition.
\begin{defn}[{\S 3.3.1 of \cite{HennRes}}]\label{defn:resofspectra}
A \emph{resolution of spectra} is a sequence of spectra
\[ \xymatrix{ \ast \ar[r] &  X_{-1} \ar[r] &  X_{0} \ar[r] &  X_{1} \ar[r]   & \cdots  }\]
such that 
\begin{enumerate}[(1)]
\item the composite of any two consecutive maps is null, and
\item for every $i\geq 0$,  the map $X_i \to X_{i+1}$ factors as 
\[X_i \to C_i \to X_{i+1}\]
 such that, letting $C_{-1}:=X_{-1}$, the sequences $C_{i-1} \to X_i \to C_i$ are cofiber sequences. 
\end{enumerate}
We say that the resolution has \emph{length $n$} if $X_n \simeq C_n$ and $X_i\simeq \ast$ for every $i>n$.
\end{defn}

The first step to constructing a resolution of spectra from the algebraic duality resolution is to produce a corresponding sequence in the category of Morava modules. We begin by noting that any $P\GG_2^1$-module is a $\GG_2^1$ module by restricting along the quotient. Letting $M_0=\W$ and $\partial_0 = \varepsilon$, the sequences \eqref{eq:M1}, \eqref{eq:M2} and \eqref{eq:M3}  give corresponding short exact sequences
\begin{align}\label{eq:sesforG21}
\xymatrix{0 \ar[r] &  M_{p+1} \ar[r] &  \sD_p \ar[r]^-{\partial_p} &  M_p \ar[r] & 0 }, \quad \quad 0\leq p\leq 2
\end{align}

Recall from Proposition 4.1 of \cite{HennCent} that if $\phi \colon G \to \Gal$  is a homomorphism with kernel $S$, and $P$ is a Galois-twisted $p$-profinite $G$-module, then if $P$ is $\W[\![S]\!]$-projective, then it is $\W[\![G]\!]$-projective. Applying this to $\phi$ the identity from $\Gal$ to itself, it allows us to conclude that Galois-twisted $p$-profinite projective $\W$-modules are $\W[\Gal]$-projective.

\begin{lemma}\label{lem:ses-split}
Each of the short exact sequences in \eqref{eq:sesforG21} is a split short exact sequence of projective $\WW[\Gal]$-modules.
\end{lemma}
\begin{proof}
In the short exact sequence for $p=0$, the middle and right terms are projective $\WW[\Gal]$-modules (since they are projective $\W$-modules). So, the sequence splits and the left term $M_1$ is also projective, because it is a direct summand of the projective module $\sD_0$. Now repeat this argument inductively in $p$ to arrive at the conclusion.
\end{proof}

Recall (\cite[Def. 1.4.1]{Weibel}) that a complex $(C_n, d_n \colon C_n \to C_{n-1})$ is called split exact if there exist maps $s_n\colon C_{n} \to C_{n+1}$ such that $d_n=d_ns_{n-1}d_n$. It is straightforward to show that if a long exact sequence is spliced from split short exact sequences, then it is split exact:  Splitting maps can be taken be the compositions of splitting maps for the short exact sequences.

\begin{corollary}
The long exact sequence of \Cref{thm:updatedualres} 
\begin{equation}\label{eq:alg-res}
\xymatrix{
0 \ar[r] & \sD_3 \ar[r]^-{\partial_3} & \sD_2  \ar[r]^-{\partial_2} & \sD_1   \ar[r]^-{\partial_1} & \sD_0 \ar[r]^-{\varepsilon}  & \W \ar[r] & 0 }
 \end{equation}
where $\sD_0  \cong\W {\uparrow}^{\G_{2}^1}_{G_{48}} $, $\sD_1 \cong\sD_2 \cong\W {\uparrow}^{\G_{2}^1}_{G_{12}} $ and  $\sD_3  \cong\W {\uparrow}^{\G_{2}^1}_{G_{48}'} $
is  split exact as a sequence of projective $\WW[\Gal]$-modules.
\end{corollary}

Next, we induce the short exact sequences \eqref{eq:sesforG21} to sequences of $\W[\![\G_2]\!]$-modules by applying the functor $\W[\![\G_2]\!]\otimes_{\W[\![\G_2^1]\!]}(-)$ from the category $\Mod^\phi(\GG_2^1)$ to $\Mod^\phi(\GG_2)$. We denote $N_i=\W[\![\G_2]\!]\otimes_{\W[\![\G_2^1]\!]} M_i$ and introduce notation for the maps as follows:
\begin{gather}\label{eq:inducedSES}
\begin{aligned} 
\xymatrix{0 \ar[r] &  N_1  \ar[r]^-{\beta_1} & \W {\uparrow}^{\G_{2}}_{G_{48}}  \ar[r]^-{\alpha_0} & \W{\uparrow}^{\G_{2}}_{\G_{2}^1} \ar[r] &   0}\\
\xymatrix{0 \ar[r] &   N_2 \ar[r]^-{\beta_2}  & \W {\uparrow}^{\G_{2}}_{G_{12}}   \ar[r]^-{\gamma_1}  & N_1 \ar[r] &   0}\\
\xymatrix{0 \ar[r] & \W {\uparrow}^{\G_{2}}_{G_{48}}  \ar[r]^-{\alpha_3} &  \W {\uparrow}^{\G_{2}}_{G_{12}}  \ar[r]^-{\gamma_2} &  N_2 \ar[r] &   0}
\end{aligned}
\end{gather}

Then we have the split exact sequence 
\begin{equation}\label{eq:alg-res-2}
 0 \to \W {\uparrow}^{\G_{2}}_{G_{48}'} \xra{\alpha_3}\W {\uparrow}^{\G_{2}}_{G_{12}} \xra{\alpha_2}\W {\uparrow}^{\G_{2}}_{G_{12}}   \xra{\alpha_1}\W {\uparrow}^{\G_{2}}_{G_{48}}  \xra{\alpha_0} \W{\uparrow}^{\G_{2}}_{\G_{2}^1} \to 0 
 \end{equation}
which is obtained by splicing the short exact sequences \eqref{eq:inducedSES} and where $\alpha_i= \beta_i \circ \gamma_i$ for $i=2,3$.

\begin{theorem}\label{cor:Ehomologystart}
There exists a split exact sequence of Morava modules spliced from split short exact sequences
\begin{equation} \label{eq:MM}
\xymatrix@C=1pc{
&&&&P_2 \ar[rd]^{\gamma_2} &\\
E_*E^{h\GG_2^1} \ar[r]^{\widetilde\alpha_0}_\subset &E_*E^{hG_{48}} \ar[rd]_{\widetilde\beta_1} \ar[rr]^{\widetilde\alpha_1} &&E_*E^{hG_{12}}\ar[rr]^{\widetilde\alpha_2} \ar[ru]^{\widetilde\beta_2} &&E_*E^{hG_{12}}\ar@{->>}[rr]^{\widetilde\alpha_3} \ar[dr]_-{\widetilde\beta_3} &&E_*E^{hG_{48}}\\
&&P_1 \ar[ru]_{\widetilde\gamma_1} &&  && P_3 \ar[ur]_{\widetilde\gamma_3}^-{\cong}
 }
\end{equation}
where $P_i=\Hom_{\WW[\Gal]}(N_i, E_*)$ and $\widetilde\alpha_i$, $\widetilde\beta_i$ and $\widetilde\gamma_i$ are duals of the maps with the same names in \eqref{eq:inducedSES} and \eqref{eq:alg-res-2}.
\end{theorem}

\begin{proof}
We apply the functor $\Hom_{\W[\Gal]}(-, E_*)$ of continuous homomorphisms of Galois-twisted $\Gal$-modules, to the sequences of \eqref{eq:inducedSES} and \eqref{eq:alg-res-2}. We use the fact that $\Hom_{\WW[\Gal]}(-, E_*)$ is an additive functor, hence it preserves split exact sequences.  
Finally, we use the fact, see  \cite[Corollary 5.5]{HennCent}, that for any closed subgroup $G\subseteq \G_2$ 
there is an isomorphism of Morava modules
\[
E_*E^{hG} \cong \Hom_{\W[\Gal]}(\W{\uparrow}^{\mathbb{G}_2 }_{G}  , E_*)
\]
to finish the proof.  Note that for the last term, we use that $E^{hG_{48}} \simeq E^{hG_{48}'}$, and from this point onward, the  difference between $G_{48}$ and $G_{48}'$ is no longer relevant.
\end{proof}

Now we would like to realize the resolution of Morava modules \eqref{eq:MM} as a resolution of spectra.  We will need the following lemma.
\begin{lemma}\label{lem:Hurewicz}
Let $E=E_2(\Gamma)$ be Morava $E$-theory and $\mathcal{E}\mathcal{G}_2$ the associated category of Morava modules. Let $X_{-1} = E^{h\GG_2^1}$, $X_0=X_3=E^{hG_{48}}$,  $X_1=X_2=E^{hG_{12}}$. Then
\begin{enumerate}[(1)]
\item  the $E$-Hurewicz homomorphisms
\[
\pi_0 F(X_i , X_{i+1}) \to  \Hom_{\mathcal{E}\mathcal{G}_2}(E_*X_i , E_*X_{i+1})
\]
are surjective for $i=-1,0,1,2$, and
\item  the $E$-Hurewicz homomorphisms
\[
\pi_0 F(X_i , X_{i+2}) \to  \Hom_{\mathcal{E}\mathcal{G}_2}(E_*X_i , E_*X_{i+2})
\]
are 
 injective for $i=-1,0,1$.
\end{enumerate}
\end{lemma}
\begin{proof}
Except for the case when $i=-1$, this is shown, and stated, in the proof of Proposition 5.11 of \cite{HennCent} (see (5.2) and (5.3)). However, the same arguments given in this proof also apply in the case $i=-1$.
\end{proof}

Now we proceed as in \cite[Lem. 3.23]{BobkovaGoerss}. We truncate the resolution of \Cref{cor:Ehomologystart}
\begin{equation}\label{eq:truncated-alg}
\xymatrix{E_*E^{h\G_2^1} \ar[r]^-{\alpha_0}  & E_*E^{hG_{48}} \ar[r]^-{\alpha_1}  & E_*E^{hG_{12}} \ar[r]^-{\alpha_2} &E_*E^{hG_{12}}}
\end{equation}
and realize this truncated resolution topologically.

\begin{lemma}\label{lem:truncated}
There exists a sequence of spectra
\begin{equation}\label{eq:truncated}
E^{h\G_2^1} \xra{\overline{\alpha}_0} E^{hG_{48}} \xra{\overline{\alpha}_1} E^{hG_{12}} \xra{\overline{\alpha}_2} E^{hG_{12}}
\end{equation}
whose $E_*$-homology is \eqref{eq:truncated-alg}, and such that
\begin{enumerate}
\item\label{it:first} the compositions $\overline{\alpha}_1\overline{\alpha}_{2}$ and $\overline{\alpha}_0\overline{\alpha}_{1}$ are null-homotopic, and 
\item\label{it:second} the first map $\overline{\alpha}_0$ induces a surjective homomorphism 
\[
\pi_* F(E^{hG_{48}}, E^{hG_{12}}) \twoheadrightarrow \pi_* F(E^{h\mathbb{G}_2^1}, E^{hG_{12}}).
\]
\end{enumerate}
\end{lemma}
\begin{proof}
We define the map $\overline{\alpha}_0\colon E^{h\GG_2^1} \to E^{hG_{48}}$ to be the inclusion of fixed points. Then we apply the surjectivity statements in  \Cref{lem:Hurewicz} to define the maps $\overline{\alpha}_1$ and $\overline{\alpha}_2$ and the injectivity statements to show that $\overline{\alpha}_1 \circ \overline{\alpha}_{2} \simeq 0$  and $\overline{\alpha}_0 \circ \overline{\alpha}_{1} \simeq 0$.
Part \eqref{it:second} is proved as in the last part of the proof of \cite[Lemma 3.23]{BobkovaGoerss}.
\end{proof}

\begin{corollary}\label{cor:tower-fib}
The sequence \eqref{eq:truncated} gives rise to a diagram
\begin{equation}\label{eq:tower}
\xymatrix@=1.5pc{
E^{h\GG_2^1} \ar[rr]^{\overline{\alpha}_0} \ar@{=}[dr] & &E^{hG_{48}} \ar[rd]_{\overline{\beta}_1} \ar[rr]^{\overline{\alpha}_1} &&E^{hG_{12}}\ar[rr]^{\overline{\alpha}_2} \ar[dr]_{\overline{\beta}_2} &&E^{hG_{12}} \ar[dr]_-{\overline{\beta}_3}\\
&C_0\ar[ur]_-{\bar{\alpha}_0}  &&C_1\ar[ru]_{\overline{\gamma}_1} \ar@{-->}[ll]^-{\delta_1}  &&C_2 \ar[ur]_{\overline{\gamma}_2}  \ar@{-->}[ll]^-{\delta_2}   &&C_3  \ar@{-->}[ll]^-{\delta_3} 
 }
\end{equation}
where $C_1$ is the cofiber  of $\overline{\alpha}_{0}$, $C_2$ is the cofiber of  $\overline{\gamma}_1$, $C_3$ is the cofiber of $\overline{\gamma}_2$ and $\overline{\alpha}_i=\overline{\beta}_i\overline{\gamma}_i$, and $\delta_i \colon C_{i}\to \Sigma C_{i-1}$ is the next map in the distinguished triangle.
Furthermore,
 $E_*C_i =P_i$ and $E_*\overline{\beta}_i=\beta_i, E_*\overline{\gamma}_i=\gamma_i$. 
\end{corollary}

\begin{proof}
Define $C_1$ as the cofiber of $\overline{\alpha}_0$. Since $\overline{\alpha}_1\overline{\alpha}_0$ is nullhomotopic, we  get a map $\overline{\gamma}_1$ so that $\overline{\alpha}_1 = \overline{\gamma}_1\overline{\beta}_1 $. Define $C_2$ as the cofiber of $\overline{\gamma}_1$. 
We would like to show that $\overline{\alpha}_2 \overline{\gamma}_1\simeq 0$ in order to produce $\overline{\gamma}_2$. Consider the diagram where the top row is an unraveled fiber sequence 
\[
\xymatrix{
E^{h\mathbb{G}_2^1} \ar[r]^{\overline{\alpha}_0}  &E^{hG_{48}} \ar[d]_{\overline{\alpha}_1}\ar[r]^{\overline{\beta}_1} &C_1 \ar[r] \ar[dl]_{\overline{\gamma}_1} &\Sigma E^{h\mathbb{G}_2^1} \ar@{-->}[ddll]_{\widetilde{\gamma}_1} \ar[r]^{\Sigma \overline{\alpha}_0} &\Sigma E^{hG_{48}} \ar@{..>}[ddlll]^{\widetilde{\widetilde{\gamma}}_1} \\
& E^{hG_{12}} \ar[d]_{\overline{\alpha}_2} \\
& E^{hG_{12}}
}
\]
Since $\overline{\alpha}_2 \overline{\alpha}_1\simeq 0$, then $\overline{\alpha}_2 \overline{\gamma}_1\overline{\beta}_1\simeq 0$, so the map $\overline{\alpha}_2 \overline{\gamma}_1$ factors through  a map
\[\widetilde{\gamma}_1\colon\Sigma E^{h\mathbb{G}_2^1} \to E^{hG_{12}}.\] 
But by \Cref{lem:truncated}(2), we have that $\widetilde{\gamma}_1$ factors through a map 
\[\widetilde{\widetilde{\gamma}}_1 \colon \Sigma E^{hG_{48}} \to E^{hG_{12}},\] 
hence $\overline{\alpha}_2\overline{\gamma}_1 \simeq 0$ by \cref{lem:Hurewicz}.

Now consider the middle cofiber sequence in the diagram \eqref{eq:tower}
\[
C_1 \xra{\overline{\gamma}_1} E^{hG_{12}} \xra{\overline{\beta}_2} C_2 
\]
Since $\overline{\gamma}_1 \overline{\alpha}_2 \simeq 0$, there exists a factorization $\overline{\gamma}_2$ as in diagram \eqref{eq:tower}.  

The last claim follows from the fact that $E_*\overline{\alpha}_i=\alpha_i$.
\end{proof}

From \Cref{cor:tower-fib} it follows that  there is a resolution of spectra 
\[
E^{h\G_2^1}\to E^{hG_{48}} \to E^{hG_{12}} \to E^{hG_{12}} \to C_3.
\]
Equivalently, there is a tower of fibrations
\begin{align}\label{eq:toweroffib}
\xymatrix{
E^{h\GG_2^1} \ar[r] & Z_2 \ar[r] & Z_1 \ar[r] & E^{hG_{48}} \\
\Sigma^{-3}C_3 \ar[u]^-{i} & \Sigma^{-2}E^{hG_{12}}  \ar[u] &   \Sigma^{-1}E^{hG_{12}} \ar[u]
}
\end{align}
where $Z_{s-1}$ is related to $C_{s}$ via a cofiber sequence
\[ \Sigma^{-s}C_s \to E^{h\GG_2^1} \to Z_{s-1}. \]

In the rest of this paper, we will show that $C_3\simeq \Sigma^{48}E^{hG_{48}}$. We will be mirroring arguments of \cite{BobkovaGoerss}. To make it easier to follow that paper, we use the notation
\[X := C_3.\]
It follows from \cref{cor:tower-fib} that
 there is an isomorphism of Morava modules
 \begin{equation}\label{eq:MM-iso}
  E_* E^{hG_{48}} \xra[\cong]{\gamma_3^{-1}} P_3 =E_* X   
\end{equation} 
We then have the following result, which completes the construction of the topological duality resolution and the proof of \cref{thmintro:topres}.
 \begin{theorem}\label{thm:last}
 There is an equivalence 
\[\Sigma^{48}E^{hG_{48}} \xrightarrow{\simeq} X\]
 which induces the isomorphism of Morava modules \eqref{eq:MM-iso}.
\end{theorem}

The proof is essentially the same as that of Theorem 5.8 of \cite{BobkovaGoerss}. Now, this result is the culmination of both Sections 4 and 5 of that paper and so we give some ideas of how the adaptation goes. The penultimate result in  \cite{BobkovaGoerss} going into Theorem 5.8 is their Corollary  5.3, which we explain in our context here.

We will be interested in $E_*$-based Hurewicz homomorphisms
\[
\pi_* F(E^{hG_{48}}, X) \to  \Hom_{\mathcal{E}\mathcal{G}_2}(E_*E^{hG_{48}}, E_*E^{hG_{48}})\cong (E_*[[\mathbb{G}_2/{G_{48}}]])^{G_{48}} ,
\]
where the isomorphism follows from \cite[Prop. 2.7]{ghmr},
and
\[
\pi_* X \to \Hom_{E_*E}(E_*, E_*X) \cong  \Hom_{\mathcal{E}\mathcal{G}_2}(E_*, E_*E^{hG_{48}})\cong (E_*)^{G_{48}} .
\]

Recall that the $G_{48}$-invariants of $E_{48}$ are isomorphic to a free module over the power series ring $ \Z_2[\![j]\!]$ with generator $\Delta^2$. See, for example, Section 2.2 of \cite{BobkovaGoerss} for more details. Consider the diagram
\[\xymatrix{ \pi_{48}F(E^{hG_{48}},X)  \ar[r]^-{\iota_*} \ar[d]^-H & \pi_{48}X \ar[d]^-H \\
\Hom_{E^0[\![\mathbb{G}_2]\!]}(E^0[\![\G_2/G_{48}]\!], E^{-48}[\![\G_2/G_{48}]\!]) \ar[r]^-{\iota^*} & \Hom_{E^0[\![\mathbb{G}_2]\!]}(E^0[\![\G_2/G_{48}]\!], E^{-48})  \ar[d]^-\cong\\
 & E_{48}^{G_{48}} \cong \Z_2[\![j]\!]\Delta^2\ar[d]^-{\epsilon} \\
 & \F_2
}\]
with vertical arrows $H$ the Hurewicz maps, $\iota_*$ induced by the unit of  $E^{hG_{48}}$,  and $\epsilon$  reduction modulo $(2,j)$. 

The following result is proved exactly as in \cite{BobkovaGoerss}. 
\begin{lemma}[\cite{BobkovaGoerss}, Corollary 5.3]
For any $f(j)\Delta^2$ in $E_{48}^{G_{48}}$, there
 is a map
\[ \phi_f \colon \Sigma^{48}E^{hG_{48}} \to X \]
such that
\[ \epsilon(H\iota_*(\phi_f)) \equiv f(0) \mod 2. \]
\end{lemma}

We can now finally prove \Cref{thm:last}.
\begin{proof}[Proof of \Cref{thm:last}]
We can run the proof of Theorem 5.8 in \cite{BobkovaGoerss} adapted to our context, and the arguments are actually simpler working with $\GG_2^1$: We choose $f(j)=1$ above, giving us a map
\[\phi_1 \colon \Sigma^{48}E^{hG_{48}} \to X \]
such that  $\epsilon(H\iota_*(\phi_1)) =1$. To show that $\phi_1$ is an equivalence, it is enough to prove that it induces an isomorphism on $E^0(-)$. We apply $E^0(-)$ to this map to get
\[ 
\xymatrix{E^0(X)  \ar[r]^-{E^0(\phi_1) } \ar[d]_-\cong &  E^{-48}(E^{hG_{48}})  \ar[d]^-\cong\\
E^0[\![ \GG_2/G_{48}]\!] \ar[r] & E^{-48}[\![ \GG_2/G_{48}]\!].
}  \]
If we restrict along $K\subset \GG_2$, using the fact that $K \to \GG_2 \to \GG_2/G_{48}$ is an isomorphism of left $K$-sets, we get a map of $ E^0[\![K]\!] $-modules
\[ E^0[\![K]\!]  \to E^{-48}[\![ K]\!]
\]
and it is enough, using a form of Nakayma's Lemma, to prove that this is an isomorphism after reduction modulo the maximal ideal $\mathfrak{m}_K$ of $E^0[\![K]\!] $, which is the kernel of the composition
\[E^0[\![K]\!]  \to E^0 \to \F_4.\] 
But, consider the diagram
\[\xymatrix{
\Hom_{E^0[\![\GG_2]\!]}(E^0[\![\G_2/G_{48}]\!], E^{-48}[\![\G_2/G_{48}]\!]) \ar[r] \ar[d]^-{\iota^*}&  \Hom_{E^0[\![K]\!]}(E^0[\![K]\!], E^{-48}[\![K]\!]) \ar[d]^-{\iota^*}\\
\Hom_{E^0[\![\GG_2]\!]}(E^0[\![\G_2/G_{48}]\!], E^{-48}) \ar[r]  \ar[d]^-\cong  &  \Hom_{E^0[\![K]\!]}(E^0[\![K]\!], E^{-48})  \ar[d]^-\subset \\
E_{48}^{G_{48}} \cong \Z_2[\![j]\!] \ar[r]^-\subset \ar[d]^-\epsilon & E_{48} \cong \W[\![u_1]\!] \ar[d]^-\epsilon \\
\F_2 \ar[r] & \F_4
}\]
where the top horizontal arrow is induced by restriction along the inclusion $K\subset \G_2$. Following $E^0(\phi_1)$ from the top left to bottom right corner, we see that it maps to $1\in \F_4$. That means that, modulo $\mathfrak{m}_K$, $E^0(\phi_1)$ induces the identity map, which is clearly an isomorphism.
\end{proof}

%\bibliographystyle{alphaurl}
%\bibliography{duality-bib}

\newcommand{\etalchar}[1]{$^{#1}$}

\end{document}